\documentclass[12pt,reqno]{amsart}

\usepackage{amssymb, xcolor}
\usepackage{amscd}
\usepackage{amsfonts}
\usepackage{setspace}
\usepackage{version}
\usepackage{tikz}
\usepackage{caption}
\usepackage{hyperref}
\usepackage{comment}
\usepackage{dsfont}

\usepackage{graphicx}
\usepackage{subcaption}

\newtheorem{theorem}{Theorem}[section]
\newtheorem{proposition}[theorem]{Proposition}
\newtheorem{remark}[theorem]{Remark}
\newtheorem{problem}[theorem]{Problem}
\newtheorem{lemma}[theorem]{Lemma}

\hypersetup{
    colorlinks=true,
    linkcolor=blue,
    filecolor=magenta,      
    urlcolor=blue,
    pdftitle={Overleaf Example},
    pdfpagemode=FullScreen,
    }

\theoremstyle{definition}
\renewcommand{\leq}{\leqslant}
\renewcommand{\geq}{\geqslant}

\numberwithin{equation}{section}
\usepackage{amsmath, amssymb, amsthm, mathtools}
\usepackage[a4paper,margin=1in]{geometry}

\usepackage{color}
\usepackage{amssymb}

\begin{document}

\title{The smoothest average and some extremal problems for polynomials}
\author[J. Gait\'an]{Jos\'e Gait\'an}
\address[JG]{Department of Mathematics, Virginia Polytechnic Institute and State University,  225 Stanger Street, Blacksburg, VA 24061-1026, USA
}
\email{jogaitan@vt.edu}

\author[C. Garz\'on]{Carlos Garz\'on}
\address[CG]{Department of Mathematics, Virginia Polytechnic Institute and State University,  225 Stanger Street, Blacksburg, VA 24061-1026, USA
}
\email{cgarzongu@vt.edu}

\author[J. Madrid]{Jos\'e Madrid}
\address[JM]{Department of Mathematics, Virginia Polytechnic Institute and State University,  225 Stanger Street, Blacksburg, VA 24061-1026, USA
}
\email{josemadrid@vt.edu}


\date{}

\begin{abstract}
    We study the problem of finding the ``smoothest'' local average of a function $f \in \ell^2(\mathbb{Z})$ when we consider its convolution with suitable kernels $u$. The measurement of smoothness is as follows: Given a positive integer $k$, we aim to minimize the constant
    \begin{equation*}
        \sup_{0 \neq f \in \ell^2(\mathbb{Z})} \frac{\|\nabla^{k}(u\ast f)\|_{\ell^2(\mathbb{Z})}}{\|f\|_{\ell^2(\mathbb{Z})}}
    \end{equation*}
    among all symmetric kernels $u : \{-n,\dots,n\} \to \mathbb{R}$ with normalization $\sum_{j=-n}^{n}u(j) = 1$. We are also interested in finding the kernel for which the least constant is attained. For $k=1$ and $k=2$, the sharp constants and optimal kernels were already known \cite{KravitzSteinerberger2021,Richardson2026}. In this paper, we provide alternative proofs for $k\in \{1,2\}$ by using complex analysis tools. Moreover, we establish the case $k=3$, and also the cases $k\in \{4,6\}$ when the kernels are restricted to have non-negative Fourier transform. These are the first results in the literature for $k>2$. Finally, we deduce a general relation between the sharp constants and optimal kernels corresponding to $\nabla^k$ and $\nabla^{2k}$.
\end{abstract}

\maketitle

\section{Introduction}

Given a real-valued function $f$ defined on $\mathbb{R}^d$, a problem arising in some applications is that of finding the best way to average $f$ over a given scale. Depending on the interpretation, this question may have different answers. A common approach is to average $f$ by convolving it with a fixed kernel $u$, and then the problem becomes: What properties should we ask $u$ to have? What is the best choice of $u$ satisfying these properties?

Steinerberger \cite{Steinerberger2021} proposed to require that the
convolution at a certain scale should be as smooth as possible, and he suggested to measure the smoothness of a function by the $L^2$-norm of its derivative. As a result, the aim was to find the non-negative, radial function $u : \mathbb{R}^d \to \mathbb{R}$ with normalized $L^1$-mass and a normalized moment (i.e. $\int_{\mathbb{R}^d} u(x) dx$ and $\int_{\mathbb{R}^d} |x|^{\alpha} u(x) dx$ would be fixed) which minimizes the constant
\begin{equation*}
    C_u = \sup_{0 \neq f \in L^2(\mathbb{R}^d)} \frac{\|\nabla(u\ast f)\|_{L^2(\mathbb{R}^d)}}{\|f\|_{L^2(\mathbb{R}^d)}}.
\end{equation*}
Applying the Fourier transform in $\mathbb{R}^d$ together with Plancherel's theorem, we can see that this problem is equivalent to minimizing $\|\xi \cdot \widehat{u}(\xi)\|_{L^{\infty}(\mathbb{R}^d)}$ for such family of kernels $u$. This question led to new uncertainty principles for the Fourier transform in $\mathbb{R}$.

A discrete analogue to this problem was studied by Kravitz and Steinerberger in \cite{KravitzSteinerberger2021}. In order to restrict the scale of the average in this context, they bounded the support of $u$ instead of fixing its moment. More precisely, they were interested in finding the symmetric function $u : \{-n,\dots,n\} \to \mathbb{R}$ with normalization $\sum_{k=-n}^{n} u(k) = 1$ which minimizes the constant
\begin{equation} \label{Const:Cu}
    C_u = \sup_{0 \neq f \in \ell^2(\mathbb{Z})} \frac{\|\nabla(u\ast f)\|_{\ell^2(\mathbb{Z})}}{\|f\|_{\ell^2(\mathbb{Z})}}, 
\end{equation}
where $\nabla$ denotes the discrete derivative (which is given by $\nabla g(k) = g(k+1) - g(k)$). Similarly to the approach in the continuous setting, an application of the Fourier transform shows that this problem is equivalent to minimizing $\|(e^{i\xi}-1) \widehat{u}(\xi)\|_{L^{\infty}(\mathbb{T})}$ for this collection of kernels. Kravitz and Steinerberger \cite[Theorem 1]{KravitzSteinerberger2021} established that the constant function $u(k) = 1/(2n+1)$ is the only kernel satisfying the desired conditions that minimizes \eqref{Const:Cu}, and in such a case $C_u = 2/(2n+1)$. Their proof consists in reducing the problem to minimizing the $L^{\infty}$-norm of a class of polynomials on $[-1,1]$\footnote{It is also equivalent to minimizing $\|(1-x)^{1/2}p(x)\|_{L^{\infty}([-1,1])}$ over all polynomials $p$ of degree at most $n$ such that $p(1)=1$. It can be shown that this $L^{\infty}$-norm attains its minimum only at the normalized $n$-th Chebyshev polynomial of the fourth kind $\frac{1}{2n+1}W_n(x)$.}, and then adapting the argument that shows that the $n$-th Chebyshev polynomial minimizes $\max_{x \in [-1,1]}|p(x)|$ among all monic polynomials $p$ of degree $n$. 

Our first result in this paper provides a different proof for the sharp inequality $C_u \geq 2/(2n+1)$ when $u$ is assumed to be a ``desirable" kernel. The main difference is that we prove this result by applying tools from complex analysis, avoiding the use of Chebyshev polynomials. More precisely, we use a refined version of Bernstein's inequality for complex polynomials with certain symmetry in their coefficients (see Lemma \ref{LemmaSelfInversePol}).

\begin{proposition}\label{Thm:FirstDerivative}
    Let $u : \{-n,\dots,n\} \to \mathbb{R}$ be a symmetric function with normalization $\sum_{k=-n}^{n} u(k) = 1$. Then we have the inequality
    \begin{equation}\label{Constant:Sharp1NoRestrict}
        \sup_{0\neq f \in \ell^2(\mathbb{Z})} \frac{\|\nabla(u \ast f)\|_{\ell^2(\mathbb{Z})}}{\|f\|_{\ell^2(\mathbb{Z})}} \geq \frac{2}{2n+1}.
    \end{equation}
\end{proposition}

Kravitz and Steinerberger \cite[Theorem 2]{KravitzSteinerberger2021} also considered the case when the smoothness of a function is measured with the $\ell^2$-norm of its discrete second derivative, which is given by $\nabla^2 g(k) = \nabla(\nabla g)(k)$. They found the sharp constant for this case when the kernel $u$ is assumed to have non-negative Fourier transform, and their proof is based once again on Chebyshev polynomials. We also present in this paper a different proof of this optimal inequality, which relies on complex analysis methods. Namely, we apply both Fej\'er-Riesz theorem (see Lemma \ref{Lemma:FejerRiesz}) and Erd\"os-Lax theorem (see Lemma \ref{Lemma:ErdosLax}) in our argument.

\begin{proposition}\label{Thm:SecondDerivative}
    Let $u : \{-n,\dots,n\} \to \mathbb{R}$ be a symmetric function with normalization $\sum_{k=-n}^{n} u(k) = 1$ and non-negative Fourier transform. Then we have the inequality
    \begin{equation}\label{Constant:Sharp2Restrict}
        \sup_{0\neq f \in \ell^2(\mathbb{Z})} \frac{\|\nabla^2(u \ast f)\|_{\ell^2(\mathbb{Z})}}{\|f\|_{\ell^2(\mathbb{Z})}} \geq \frac{4}{(n+1)^2}.
    \end{equation}
\end{proposition}

Richardson \cite[Theorem 1]{Richardson2026} established the corresponding optimal inequality when the kernel $u$ is no longer assumed to have non-negative Fourier transform:
\begin{equation}\label{Sharp2DerivNoRest}
    \sup_{0\neq f \in \ell^2(\mathbb{Z})} \frac{\|\nabla^2(u \ast f)\|_{\ell^2(\mathbb{Z})}}{\|f\|_{\ell^2(\mathbb{Z})}} \geq \frac{4}{n+1} \tan\left( \frac{\pi}{4(n+1)} \right).
\end{equation}
Observe that the constant on the right-hand side of \eqref{Sharp2DerivNoRest} behaves as $\pi/(n+1)^2$ when $n \to +\infty$. This shows the improvement from $4$ to $\pi$ in the sharp constant when removing the non-negative Fourier transform restriction on $u$.

Considering the results described above, it is natural to wonder what happens if we measure the smoothness of the convolution with other differential operators. In particular, what happens if we consider the discrete derivative of order $\ell>2$, defined recursively by $\nabla^{\ell} = \nabla(\nabla^{\ell-1})$? What are the corresponding sharp inequality and optimal kernel in this case? Recall that equality in \eqref{Constant:Sharp1NoRestrict} is achieved by the normalized indicator function $u = \frac{1}{2n+1}\mathds{1}_{[-n,n]}$. Additionally, we know that equality in \eqref{Constant:Sharp2Restrict} is attained only for the triangular kernel $u(k) = (n + 1 - |k|)/(n + 1)^2$ (see \cite[Theorem 2]{KravitzSteinerberger2021}), which can be interpreted as the convolution of the normalized characteristic function $\frac{1}{n+1}\mathds{1}_{[-n/2,n/2]}$ with itself. Therefore, it is natural to guess that, under certain restrictions on the kernel $u$, the optimal constant for $\nabla^{\ell}$ is attained by convolving a normalized indicator function with itself $\ell-1$ times. In our next result, we present a family of kernels for which this happens.

\begin{proposition}\label{Thm:NonSharpEvenDerivatives}
    Let $\ell$ be a positive even integer, let $n$ be a positive multiple of $\ell/2$, and let $u : \{-n,\dots,n\} \to \mathbb{R}$ be a symmetric function with normalization $\sum_{k=-n}^{n} u(k) = 1$. Assume that there exists a polynomial $Q(z) = \sum_{k=0}^{m} b_k z^k$ with no zeros in $\mathbb{D}$ such that $\widehat{u}(\xi) = |Q(e^{i\xi})|^{\ell}$ for all $\xi \in \mathbb{T}$. Then we have the inequality
    \begin{equation*}
        \sup_{0\neq f \in \ell^2(\mathbb{Z})} \frac{\|\nabla^{\ell}(u \ast f)\|_{\ell^2(\mathbb{Z})}}{\|f\|_{\ell^2(\mathbb{Z})}} \geq \left( \frac{2}{1 + 2n/\ell} \right)^{\ell}.
    \end{equation*}
\end{proposition}

The proof of this proposition essentially follows the same ideas as the proof of Proposition \ref{Thm:SecondDerivative}. Hence, we omit the details to avoid repetition. Observe that if $\ell=2$, an application of Fej\'er-Riesz theorem shows that the assumption of Proposition \ref{Thm:NonSharpEvenDerivatives} on the kernel $u$ is equivalent to requiring that $u$ has non-negative Fourier transform, which corresponds to the hypothesis in Proposition \ref{Thm:SecondDerivative}. Heuristically, the reason for this relation is that all roots of $\widehat{u}$ in $\mathbb{T}$ have even multiplicity when $\widehat{u} \geq 0$, and this allows us to write $\widehat{u} = |Q|^2$ for some suitable polynomial $Q$ with no zeros in $\mathbb{D}$. However, if we wanted a representation $\widehat{u} = |Q|^{4}$, all the roots of $\widehat{u}$ in $\mathbb{T}$ would necessarily have multiplicity divisible by $4$, which would need more assumptions on $u$. This suggests that the sharp constant in the case of four or six derivatives with $\widehat{u} \geq 0$ should be less than that arising from the convolution of four or six indicator functions, as we are about to see.

The first main result of this paper provides the sharp inequality and optimal kernel when the smoothness of an average is measured via the $\ell^2$-norm of its fourth discrete derivative and the kernels are restricted to have non-negative Fourier transform. It is worth pointing out that this is the first result of this type for differential operators with order greater than two.

\begin{theorem} \label{Thm:FourthDerivativeRest}
    Let $u : \{-n,\dots,n\} \to \mathbb{R}$ be a symmetric function with normalization $\sum_{k=-n}^{n} u(k) = 1$ and non-negative Fourier transform. Then we have the inequality
    \begin{equation} \label{Sharp4Deriv}
        \sup_{0\neq f \in \ell^2(\mathbb{Z})} \frac{\|\nabla^{4}(u \ast f)\|_{\ell^2(\mathbb{Z})}}{\|f\|_{\ell^2(\mathbb{Z})}} \geq \frac{2^6}{(n+2)^2} \tan^2 \left( \frac{\pi}{2n+4} \right).
    \end{equation}
    Moreover, the equality is attained if and only if
    \begin{equation}\label{OptimalKernel4Deriv}
        u(m) = u(-m) = \frac{1}{\pi} \int_{-1}^{1}S_{n}(x) T_m(x) \frac{dx}{\sqrt{1-x^2}}
    \end{equation}
    for every $m \in \{0,\dots,n\}$, where $S_n$ is defined as in \eqref{PolynomS_n} and $T_m$ denotes the Chebyshev polynomial of degree $m$.
\end{theorem}
\begin{remark}
    The constant on the right-hand side of \eqref{Sharp4Deriv} behaves asymptotically as 
    \begin{equation*}
        \frac{2^4 \pi^2}{(n+2)^4}
    \end{equation*}
    when $n \to +\infty$. Since $\pi^2 < 2^4$, this result clearly improves that of Proposition \ref{Thm:NonSharpEvenDerivatives} for $\ell=4$.
\end{remark}
\begin{figure}[ht]
    \centering
    \begin{subfigure}{0.48\textwidth}
        \includegraphics[width=\textwidth]{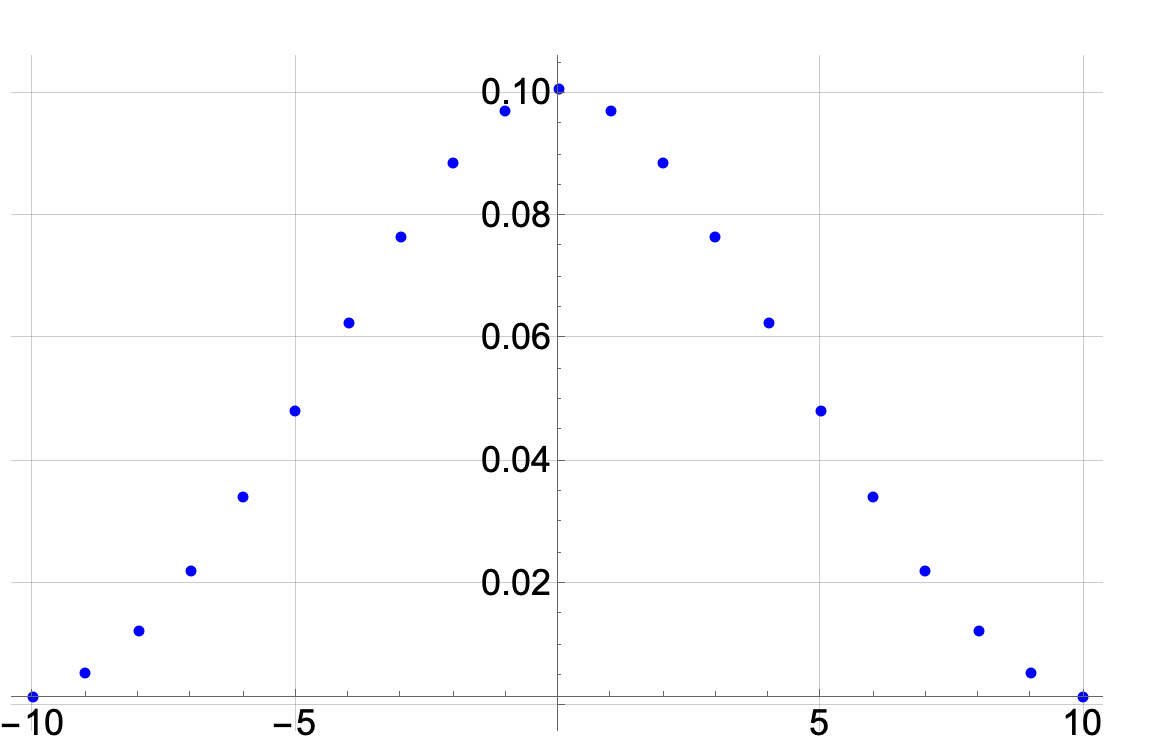}
        \caption{$u_{10}$}
        \label{fig:subfig3}
    \end{subfigure}
    \hspace{0.2cm}
    \begin{subfigure}{0.48\textwidth}
    \includegraphics[width=\textwidth]{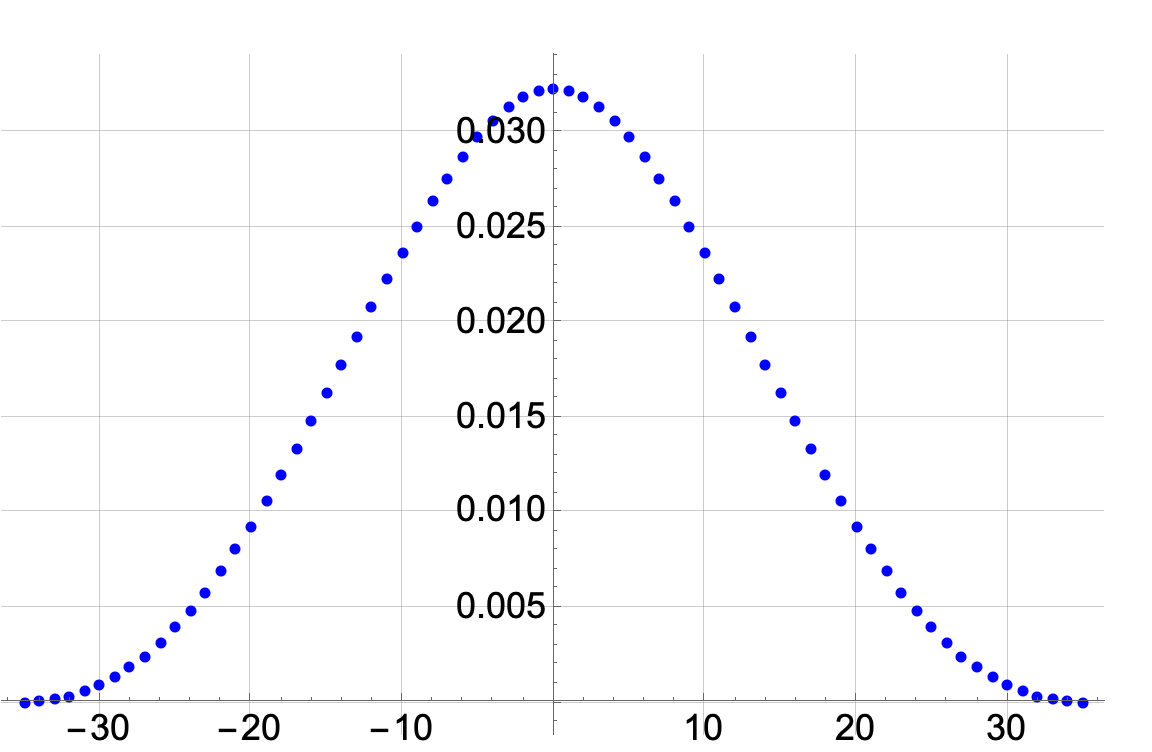}
        \caption{$u_{35}$}
        \label{fig:subfig4}
    \end{subfigure}
    \caption{The optimal kernels $u_{10}$ and $u_{35}$ from Theorem \ref{Thm:FourthDerivativeRest}.}
    \label{fig:two-subfigures1}
\end{figure}

In order to prove Theorem \ref{Thm:FourthDerivativeRest}, we apply the Fourier transform to reduce the problem to minimizing the norm $\|q\|_{L^{\infty}([-1,1])}$ over all polynomials $q$ of degree at most $n$ which have a double root at $x=1$, satisfy $q''(1) = 2$, and do not take negative values on $[-1,1]$. The extremal polynomial of this problem is obtained by manipulating properly the $n$-th Chebyshev polynomial. This result is stated precisely below.

\begin{proposition}\label{Prop:SquarePolIneq}
    Let $p$ be a polynomial of degree at most $n-2$ that is non-negative on $[-1,1]$ and satisfies $p(1)=1$. Then we have the inequality
    \begin{equation} \label{PolynomIneq}
        \max_{x \in [-1,1]}|(1-x)^2 p (x)| \geq \frac{16}{n^2} \tan^2 \left( \frac{\pi}{2n} \right).
    \end{equation}
    Moreover, the equality is attained if and only if
    \begin{equation} \label{PolynomS_n}
        p(x) = S_{n-2}(x) := \frac{8}{n^2} \tan^2 \left( \frac{\pi}{2n} \right) \frac{1}{(1-x)^2} \left( 1 + T_n\left( \frac{1+\cos(\pi/n)}{2}(x+1)-1 \right) \right).
    \end{equation}
\end{proposition}
The idea behind the construction of $S_{n-2}$ is as follows: We translate and compose $T_n$ with a linear function in such a way that
\begin{enumerate}
\item The transformed polynomial $\bar{T_n}(x):= 1 + T_n\left( \frac{1+\cos(\pi/n)}{2}(x+1)-1 \right) $ takes the two extremal values 0 and $2$ a total of $n$ times on the interval $[-1,1]$ at the points $x_k=\frac{1-\cos(\pi/n)}{1+\cos(\pi/n)}+\frac{2\cos(k\pi/n)}{1+\cos(\pi/n)}$ for all $k \in \{1, \dots, n\}$ (notice that $x_1=1$ and $x_n=-1$). But most importantly, it preserves the equioscillation property as much as possible in $[-1,1]$ with the imposed behavior at $x_1=1$. See Figure \ref{fig:1}.
\item  The polynomial $\bar{T_n}$ has a double root at $x_1=1$, which implies that the function $S_{n-2}$ is indeed a polynomial of degree $n-2$.
\item  The polynomial $S_{n-2}(x)$ is non-negative on $[-1,1]$ and satisfies the normalization $S_{n-2}(1)=1$, because $\frac{\bar{T_n}''(1)}{2!}=\frac{n^2}{8} \cot^2\left( \frac{\pi}{2n} \right)$.
\end{enumerate}
\begin{figure}[h]
    \centering
\includegraphics[scale=0.5]{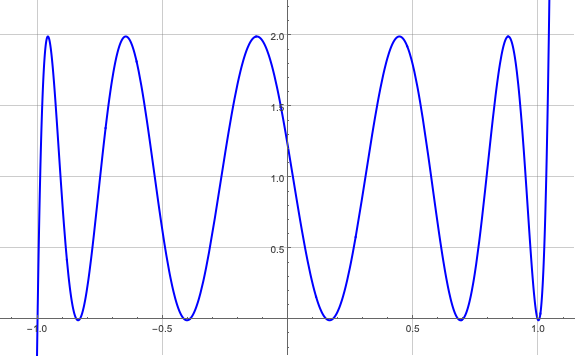}    
    \caption{Transformed Chebyshev polynomial $\bar{T}_{11}(x)$.}
    \label{fig:1}
\end{figure}

The second main result of this paper provides the sharp inequality and optimal kernel when the smoothness of an average is measured via the $\ell^2$-norm of its sixth discrete derivative and the kernels are restricted to have non-negative Fourier transform. It is worth mentioning that elliptic functions will be used for constructing the family of polynomials that will solve the problem with six discrete derivatives. Thus, some of these functions will show up in this optimal constant. 

For each $1\leq j\leq 4$, let $\Theta_j(u) := \Theta_j(u,k)$ be the $j$-th \textit{Jacobi Theta function} with \textit{Jacobi elliptic modulus} $k$, $0<k^2<1$. Let $Z(u):=Z(u,k)=\Theta_4'(u) / \Theta_4(u)$ be the  \textit{Jacobi Zeta function}. Let $\operatorname{cn}(u):=\operatorname{cn}(u,k)$, $\operatorname{sn}(u):=\operatorname{sn}(u,k)$, and $\operatorname{dn}(u):=\operatorname{dn}(u,k)$ be the \textit{Jacobi elliptic cosine, sine, and delta functions}. Also, let $K(k)$ and $E(k)$ be the \textit{complete elliptic integrals of the first and second kinds}, respectively. For references, see for example \cite[{\href{https://dlmf.nist.gov/20.2}{T. F.}},{\href{https://dlmf.nist.gov/22}{Jacobi E. F.}, \href{https://dlmf.nist.gov/19.2\#ii}{E. I.}}]{NIST:DLMF}. Our second main theorem is stated below.

\begin{theorem}\label{Thm:SixthDerivativeRest}
    Let $u : \{-n,\dots,n\} \to \mathbb{R}$ be a symmetric function with normalization $\sum_{k=-n}^{n} u(k) = 1$ and non-negative Fourier transform. Then we have the inequality
    \begin{equation} \label{Sharp6Deriv}
        \sup_{0\neq f \in \ell^2(\mathbb{Z})} \frac{\|\nabla^{6}(u \ast f)\|_{\ell^2(\mathbb{Z})}}{\|f\|_{\ell^2(\mathbb{Z})}} \geq \dfrac{2^4\cdot 3^2}{(n+3)^2}\left(\frac{1-\operatorname{cn}(2a_{n+3})}{\operatorname{dn}(2a_{n+3})}\right)^2,
    \end{equation}
    where $a_n = K(k_n)/n$ and $k_n$ is the solution to the equation
\begin{equation}\label{c=1}
    \operatorname{cn}\left(\frac{2K(k_n)}{n}\right)+2\operatorname{sn}\left(\frac{2K(k_n)}{n}\right)Z\left(\frac{K(k_n)}{n}\right)=1.
\end{equation}
Moreover, the equality is attained if and only if
    \begin{equation*}
        u(m) = u(-m) = \frac{1}{\pi} \int_{-1}^{1} \bar{Z2}_{n+3}(x,k_{n+3},1) T_{m}(x) \frac{dx}{\sqrt{1-x^2}}
    \end{equation*}
    for every $m \in \{0,\dots,n\}$, where $T_m$ denotes the Chebyshev polynomial of degree $m$ and $\bar{Z2}_{n}(x,k_n,1)$ is defined as
    \begin{equation}\label{poly6deriv}
        \bar{Z2}_n(x,k_n,1) := \dfrac{9}{n^2}\left(\frac{1-\operatorname{cn}(2a_n)}{\operatorname{dn}(2a_n)}\right)^2\dfrac{1}{(1-x)^3} \ (1-Z2_{n}(x,k_n,1)).
    \end{equation}
    Here, $Z2_n$ stands for the Zolotarev polynomial of the first kind and second type of degree $n$ (as defined in \cite{Levedev1994}).
\end{theorem}

\begin{figure}[h]
    \centering
    \begin{subfigure}[b]{0.48\textwidth}
        \includegraphics[width=\textwidth]{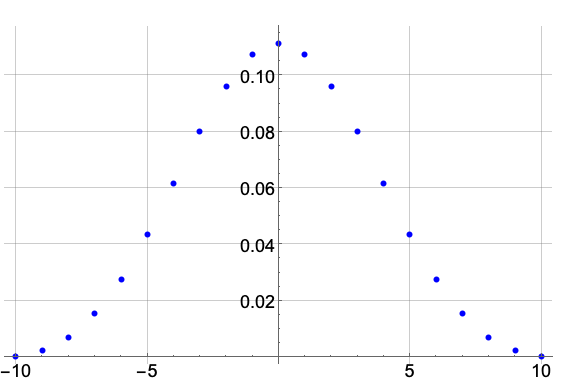}
        \caption{$u_{10}$}
        \label{fig:subfig1}
    \end{subfigure}
    \hspace{0.2cm}
    \begin{subfigure}[b]{0.48\textwidth}
    \includegraphics[width=\textwidth]{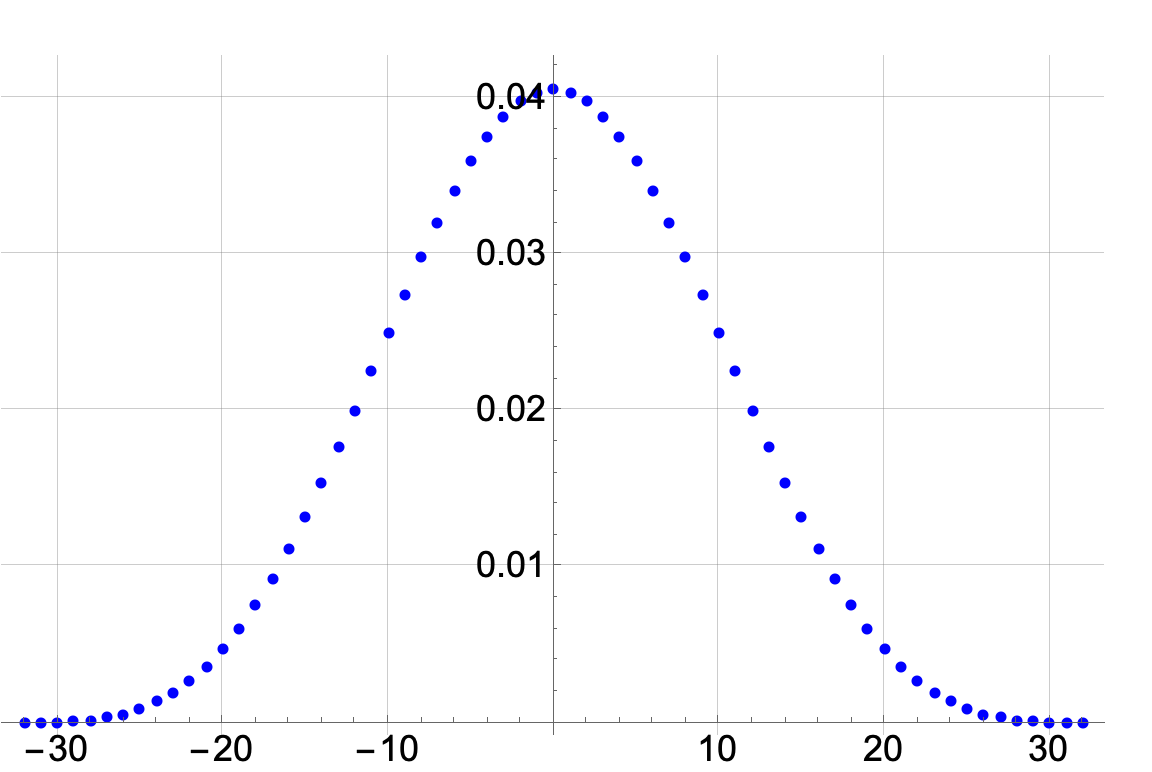}
        \caption{$u_{35}$}
        \label{fig:subfig2}
    \end{subfigure}
    
    \caption{The optimal kernels $u_{10}$ and $u_{35}$ from Theorem \ref{Thm:SixthDerivativeRest}.}
    \label{fig:two-subfigures}
\end{figure}

The following example (see Figure \ref{ExampleWaterLvl}) appears in \cite[Figure 2]{Richardson2026} as an application to smooth the data measuring the water level of Lake Chelan over the period of time from July 16, 2011 to July 30, 2011. We repeat this experiment here, but now using the kernel $u_{35}$ from Theorem \ref{Thm:SixthDerivativeRest}.

\begin{figure}[h]
    \centering
    \includegraphics[width=1\linewidth]{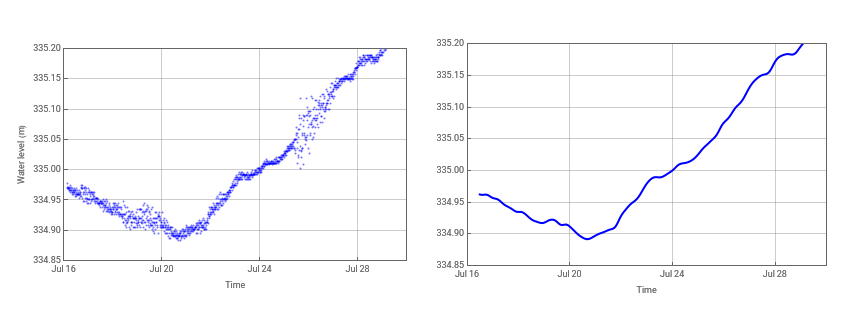}
    \caption{Original data $f$ and the smoothed data $u_{35}\ast f$. }
    \label{ExampleWaterLvl}
\end{figure}

\begin{remark}\label{Remark asymptotycs}
    The constant on the right-hand side of \eqref{Sharp6Deriv} behaves asymptotically as 
    \begin{equation*}
        \frac{2^6\cdot 3^2}{(n+3)^6}K(k^*)^4
    \end{equation*}
    when $n \to +\infty$. Here, $k^*$ is the solution of $K(k^*)=2E(k^*)$. We have obtained numerically that $k^*\approx 0.9089$ and $K(k^*)\approx2.3210$.  Table \ref{tab:comparison-40-50} shows that the asymptotic behavior is already noticeable in the range $75\leq n\leq 80$.
\begin{table}[h]
\centering
\begin{tabular}{|c| c| c|}
\hline
$n$ &
$\left(\frac{1-\operatorname{cn}(2a_n,k_n)}{\operatorname{dn}(2a_n,k_n)}\right)^2$  &
$4\left(\frac{K(0.9089)}{n}\right)^4$ \\ [10pt]
\hline
$75$  & $3.67363\times 10^{-6}$& $3.66873\times 10^{-6}$ \\
$76$  & $3.48395\times 10^{-6}$ & $3.47941\times 10^{-6}$ \\
$77$  & $3.30636\times 10^{-6}$ & $3.30216\times 10^{-6}$ \\
$78$  & $3.13994\times 10^{-6}$ & $3.13604\times 10^{-6}$ \\
$79$  & $2.98386\times 10^{-6}$ & $2.98025\times 10^{-6}$ \\
$80$  & $2.83736\times 10^{-6}$& $2.83400\times 10^{-6}$ \\
\hline
\end{tabular}
\caption{Asymptotics of the optimal constant from Theorem \ref{Thm:SixthDerivativeRest}.}
\label{tab:comparison-40-50}
\end{table}
\end{remark}
As it was done before, the proof of Theorem \ref{Thm:SixthDerivativeRest} may be reduced to minimizing the norm $\|q\|_{L^{\infty}([-1,1])}$ over all polynomials $q$ of degree at most $n$ which have a triple root at $x=1$, satisfy $q'''(1) = -6$, and do not take negative values on $[-1,1]$. However, this problem differs significantly from Proposition \ref{Prop:SquarePolIneq} (and also from \cite[Theorems 4 and 5]{KravitzSteinerberger2021} and \cite[Theorem 3]{Richardson2026}). In this situation, we need to consider the \textit{Zolotarev polynomials}, a broader class of polynomials containing those of Chebyshev. These polynomials have been successfully used to solve some extremal problems, see \cite{ErdosSzego1942,Peherstorfer1991,Levedev1994,Peherstorfer2009,BojanovNaidenov2010} and references therein. In our case, we find that the optimal polynomial is a transformation of the $n$-th degree \textit{Zolotarev polynomial of the first kind and second type}, as defined by Lebedev in \cite{Levedev1994}.

\begin{proposition}\label{(1-x)^3 theorem}
    Let $p$ be a polynomial of degree at most $n-3$ that is non-negative on $[-1,1]$ and satisfies $p(1)=1$. Then we have the inequality
    \begin{equation} \label{PolynomIneq6}
        \max_{x \in [-1,1]}|(1-x)^3 p(x)| \geq \dfrac{2\cdot3^2}{n^2}\left(\frac{1-\operatorname{cn}(2a_n)}{\operatorname{dn}(2a_n)}\right)^2,
    \end{equation}
where $a_n$ and $k_n$ are defined as in the statement of Theorem \ref{Thm:SixthDerivativeRest}. Moreover, the equality is attained if and only if $p(x) = \bar{Z2}_n(x,k_n,1)$, as defined in \eqref{poly6deriv}.
\end{proposition}

Observe that Theorem \ref{Thm:SixthDerivativeRest} follows from Proposition \ref{(1-x)^3 theorem} exactly as Theorem \ref{Thm:FourthDerivativeRest} follows from Proposition \ref{Prop:SquarePolIneq} (see Section \ref{SectionFourDerivatives}). Hence, the proof of Theorem \ref{Thm:SixthDerivativeRest} will be omitted to avoid repetition, and we shall only focus on proving Proposition \ref{(1-x)^3 theorem}.

Now that we understand better the type of extremal polynomial that optimizes the inequality \eqref{PolynomIneq}, we see that Zolotarev polynomials of the first kind and second type with suitable choices of parameters are the right choice as optimizers of \eqref{PolynomIneq6}, see Figure \ref{fig:2}.
\begin{enumerate}
\item The transformed Zolotarev polynomial $1 - Z2_{n}(x,k_n,1)$ takes the two extremal values 0 and $2$ a total of $n-1$ times. Again, it preserves the equioscillation property as much as possible in the interval $[-1,1]$ for the imposed behavior at $x=1$.
\item  The polynomial $1 - Z2_{n}(x,k_n,1)$ has a triple root at $x=1$, which implies that the function $\bar{Z2}_n(x,k_n,1)$ is a polynomial of degree $n-3$.
\item  The polynomial $\bar{Z2}_n(x,k_n,1)$ is non-negative on $[-1,1]$ and satisfies the normalization $\bar{Z2}_n(1,k_n,1)=1$, because $\frac{Z2_n'''(1,k_n,1)}{3!}= \frac{n^2}{9}\left(\frac{\operatorname{dn}(2a_n)}{1-\operatorname{cn}(2a_n)}\right)^2$.
\end{enumerate}
\begin{figure}[h]
    \centering
\includegraphics[scale=0.5]{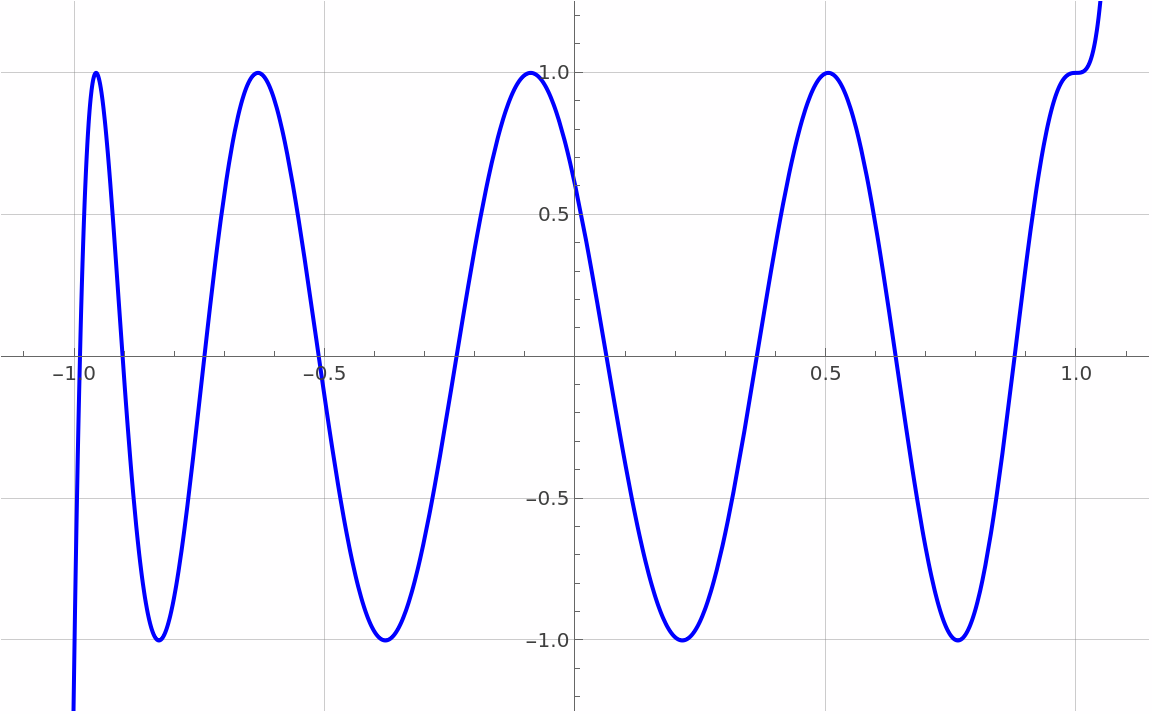}
    \caption{Zolotarev polynomial $Z2_{11}(x,k_{10},1)$, with parameter choices $k_{11}\approx 0.911718$ as in (\ref{c=1}) and $\mu=1$ that will be discussed later.}
    \label{fig:2}
\end{figure}

On the other hand, Kravitz and Steinerberger \cite{KravitzSteinerberger2021} showed that the sharp inequality for one derivative is closely related to the optimal inequality for two derivatives restricted to kernels with non-negative Fourier transform. Their argument suggests that this relation should hold for general $k$ and $2k$ derivatives. By means of Fej\'er-Riesz theorem, we can deduce the exact relation between the corresponding optimal constants.

\begin{theorem}\label{Thm:RelationK-2K}
    Let $\mathcal{F}_{n,1}$ be the family of real-valued symmetric and normalized kernels supported in $\{-n,\dots,n\}$, and let $\mathcal{F}_{n,2}$ be the family of kernels in $\mathcal{F}_{n,1}$ with non-negative Fourier transform. For each $j \in \{1,2\}$, let
    \begin{equation}\label{Cons:Ckn1}
        C_{k,n,j} = \inf_{u \in \mathcal{F}_{n,j}} \sup_{0 \neq f \in \ell^2(\mathbb{Z})} \frac{\|\nabla^k (u\ast f)\|_{\ell^2(\mathbb{Z})}}{\|f\|_{\ell^2(\mathbb{Z})}}.
    \end{equation}
    Then
    \begin{equation*}
        C_{k,n,1}^2=C_{2k,2n,2}.
    \end{equation*}
    Moreover, there exists a unique $u \in \mathcal{F}_{n,1}$ such that
    \begin{equation*}
        C_{k,n,1} = \sup_{0 \neq f \in \ell^2(\mathbb{Z})} \frac{\|\nabla^{k} (u\ast f)\|_{\ell^2(\mathbb{Z})}}{\|f\|_{\ell^2(\mathbb{Z})}},
    \end{equation*}
    and $u \ast u$ is the only kernel in $\mathcal{F}_{2n,2}$ satisfying
    \begin{equation*}
        C_{2k,2n,2} = \sup_{0 \neq f \in \ell^2(\mathbb{Z})} \frac{\|\nabla^{2k} ((u\ast u) \ast f)\|_{\ell^2(\mathbb{Z})}}{\|f\|_{\ell^2(\mathbb{Z})}}.
    \end{equation*}
\end{theorem}

\begin{remark}\label{Rmk:Richardson}
    Notice that the optimal inequality \eqref{Sharp2DerivNoRest} is recovered by combining Theorems \ref{Thm:FourthDerivativeRest} and \ref{Thm:RelationK-2K}. Additionally, we can verify explicitly that the polynomial $S_{2n-2}$ given by (\ref{PolynomS_n}) is the square of the optimal polynomial from \cite[Theorem 3]{Richardson2026}. This can be deduced from the property of the Chebyshev polynomials $T_{2n}(x)=2T_{n}^2(x)-1$ and the identity $\tan(\frac{x}{2})=\frac{\sin(x)}{1+\cos(x)}$.
\end{remark}

As an immediate consequence of Theorems \ref{Thm:SixthDerivativeRest} and \ref{Thm:RelationK-2K}, we deduce the sharp inequality for the case of three derivatives and kernels without any restriction on their Fourier transforms.

\begin{theorem}\label{Thm:ThirdDerivativeNoRest}
    Let $u : \{-n,\dots,n\} \to \mathbb{R}$ be a symmetric function with normalization $\sum_{k=-n}^{n} u(k) = 1$. Then we have the inequality
    \begin{equation} 
        \sup_{0\neq f \in \ell^2(\mathbb{Z})} \frac{\|\nabla^{3}(u \ast f)\|_{\ell^2(\mathbb{Z})}}{\|f\|_{\ell^2(\mathbb{Z})}} \geq \dfrac{2^2\cdot3}{2n+3}\left(\frac{1-\operatorname{cn}(2a_{2n+3})}{\operatorname{dn}(2a_{2n+3})}\right),
    \end{equation}
    where $a_n$ and $k_n$ are defined as in the statement of Theorem \ref{Thm:SixthDerivativeRest}. 
    Moreover, the equality is attained if and only if
    \begin{equation}\label{Kernel:SixthDerivativeRest}
        u(m) = u(-m) = \frac{1}{\pi} \int_{-1}^{1} \sqrt{\bar{Z2}_{2n+3}(x,k_{n+3},1)} T_{m}(x) \frac{dx}{\sqrt{1-x^2}}
    \end{equation}
    for every $m \in \{0,\dots,n\}$, where $T_m$ denotes the Chebyshev polynomial of degree $m$ and $\bar{Z2}_{n}(x,k_n,1)$ is defined as in (\ref{poly6deriv}).
\end{theorem}
\begin{remark}
    Notice that the function $\sqrt{\bar{Z2}_{2n+3}(x,k_{n+3},1)}$ appearing in the optimal kernel (\ref{Kernel:SixthDerivativeRest}) is still a polynomial because $\bar{Z2}_{2n+3}(x,k_{n+3},1)$ is a perfect square, since it is a polynomial of degree $2n$ with $n$ real zeros of multiplicity $2$.
\end{remark}

Taking into account the results obtained in this paper, Table \ref{Table1} summarizes the known optimal constants for measuring smoothness via the $\ell^2$-norm of $\nabla^k$, considering kernels with a non-negative Fourier transform and without any restrictions. For any $k$ not included in Table \ref{Table1}, the problem of determining the optimal inequality remains open.

\begin{table}[h]
\begin{center}
    \begin{tabular}{|c|c|c|}
    \hline
         Differential operator & Non-negative Fourier transform & No restriction\\
         \hline
         $\nabla$ & Open & \cite[Theorem 1]{KravitzSteinerberger2021} \& Proposition \ref{Thm:FirstDerivative} \\
         \hline
         $\nabla^2$ & \cite[Theorem 2]{KravitzSteinerberger2021} \& Proposition \ref{Thm:SecondDerivative} & \cite[Theorem 1]{Richardson2026} \& Remark \ref{Rmk:Richardson}\\
         \hline
         $\nabla^3$ & Open & Theorem \ref{Thm:ThirdDerivativeNoRest} \\
         \hline
         $\nabla^4$ & Theorem \ref{Thm:FourthDerivativeRest} & Open \\
         \hline
         $\nabla^6$ & Theorem \ref{Thm:SixthDerivativeRest} & Open\\
         \hline
    \end{tabular}
\end{center}
\caption{Optimal constants for derivatives of various orders.}
 \label{Table1}
\end{table}
The method developed in Section \ref{SectionSixDerivatives} allows us to find the optimal polynomial that minimizes $\|(1-x)^3p(x)\|_{L^{\infty}([-1,1])}$, which leads to the sharp constant in the case of $\nabla^{6}$. By implementing the same ideas for higher-order even derivatives, one encounters the difficulty that the analogous integral to \eqref{Elliptic integral expanded} has a weight function of the form $1/\sqrt{-G(x)}$ with $G(x)$ being a polynomial of order $N> 4$ (differing from the cases \eqref{weight function}). This problem goes beyond the Chebyshev and Zolotarev families of polynomials since the analytic solutions of the corresponding integral are expressed in terms of hyperelliptic functions.

\subsection{Organization of the paper} In Section \ref{Section1} we prove Propositions \ref{Thm:FirstDerivative} and \ref{Thm:SecondDerivative} by using tools from complex analysis. In Section \ref{SectionFourDerivatives} we show Proposition \ref{Prop:SquarePolIneq} via an equioscillation argument for polynomials, and we reduce Theorem \ref{Thm:FourthDerivativeRest} to Proposition \ref{Prop:SquarePolIneq} by means of Plancherel's theorem. In Section \ref{SectionSixDerivatives} we exploit the theory of Zolotarev polynomials to present a proof of Proposition \ref{(1-x)^3 theorem}, which immediately implies Theorem \ref{Thm:SixthDerivativeRest}. Finally, in Section \ref{SectionK-2K}, we combine Fej\'er-Riesz theorem and the results from \cite[Section 2]{NovelloSchiefermayrZinchenko2021} to deduce Theorem \ref{Thm:RelationK-2K}.

\section{Cases of one and two derivatives revisited} \label{Section1}

By applying Fourier analysis, we shall see that the proofs of \eqref{Constant:Sharp1NoRestrict} and \eqref{Constant:Sharp2Restrict} are closely related to comparing the maximum modulus of both a polynomial and its derivative in the complex unit circle. A well-known inequality due to Bernstein \cite{Bernstein1912} asserts that if $P$ is a polynomial of degree $n$, then $\max_{|z|=1}|P'(z)| \leq n \max_{|z|=1}|P(z)|$. Moreover, this inequality is optimal, and equality is attained if and only if $P(z) = \alpha z^n$ for some constant $\alpha$. In our argument, the constant $n$ in Bernstein's inequality is not good enough to deduce \eqref{Constant:Sharp1NoRestrict}. However, the symmetry of the kernels in Proposition \ref{Thm:FirstDerivative} allows us to use a refined version of Bernstein's inequality. This refined inequality is stated in the next lemma, and its proof may be found in \cite[Theorem 14.3.1]{RahmanSchmeisserBook}.

\begin{lemma} \label{LemmaSelfInversePol}
    Let $f$ be a polynomial of degree at most $N$ such that $z^N \overline{f(1/\overline{z})} \equiv e^{i\gamma} f(z)$ for some $\gamma \in \mathbb{R}$. Then,
    \begin{equation*}
        \|f'\|_{\infty} := \max_{|z|=1}|f'(z)| = \frac{N}{2} \|f\|_{\infty}.
    \end{equation*}
\end{lemma}

\begin{proof}[Proof of Proposition \ref{Thm:FirstDerivative}]
    Let $u : \{-n,\dots,n\} \to \mathbb{R}$ be a symmetric function with normalization $\sum_{k=-n}^{n}u(k) = 1$. To begin with, we shall prove that
    \begin{equation}\label{Sup1Deriv}
        \sup_{0\neq f \in \ell^2(\mathbb{Z})} \frac{\|\nabla(u \ast f)\|_{\ell^2(\mathbb{Z})}}{\|f\|_{\ell^2(\mathbb{Z})}} = \|\widehat{\nabla u}\|_{L^{\infty}(\mathbb{T})}.
    \end{equation}
    Applying Fourier transform, Plancherel's theorem, and H\"older's inequality, we get
    \begin{equation*}
        \begin{aligned}
            \|\nabla(u\ast f)\|_{\ell^2(\mathbb{Z})} &= \frac{1}{\sqrt{2\pi}}\|((\nabla u) \ast f)^{\wedge}\|_{L^2(\mathbb{T})}\\ &\leq \|\widehat{\nabla u}\|_{L^{\infty}(\mathbb{T})}\left( \frac{1}{\sqrt{2\pi}} \|\widehat{f}\|_{L^2(\mathbb{T})} \right)\\ &= \|\widehat{\nabla u}\|_{L^{\infty}(\mathbb{T})} \|f\|_{\ell^2(\mathbb{Z})}
        \end{aligned}
    \end{equation*}
    for all $f \in \ell^2(\mathbb{Z})$. As a result,
    \begin{equation*}
        \sup_{0\neq f \in \ell^2(\mathbb{Z})} \frac{\|\nabla(u \ast f)\|_{\ell^2(\mathbb{Z})}}{\|f\|_{\ell^2(\mathbb{Z})}} \leq \|\widehat{\nabla u}\|_{L^{\infty}(\mathbb{T})}.
    \end{equation*}
    Now, due to the continuity of $\widehat{\nabla u}$, we can choose $\xi_0 \in \mathbb{T}$ such that $|\widehat{\nabla u}(\xi_0)| = \|\widehat{\nabla u}\|_{L^{\infty}(\mathbb{T})}$. For each $m\geq 3$, let us consider the function $\varphi_m = (2\pi)^{1/2} |B_{1/m}(\xi_0)|^{-1/2} \mathds{1}_{B_{1/m}(\xi_0)}$, where $B_{1/m}(\xi_0)$ denotes the open ball in $\mathbb{T}$ centered at $\xi_0$ with radius $1/m$ and $\mathds{1}_{B_{1/m}(\xi_0)}$ stands for its indicator function. Notice that $\varphi_m \in L^2(\mathbb{T})$ for all $m\geq 3$. Since the Fourier transform is an isomorphism from $\ell^2(\mathbb{Z})$ to $L^2(\mathbb{T})$, for each $m\geq 3$ there exists $f_m \in \ell^2(\mathbb{Z})$ such that $\widehat{f_m} = \varphi_m$. By Plancherel's theorem, we get that $\|f_m\|_{\ell^2(\mathbb{Z})} = (2\pi)^{-1/2}\|\varphi_m\|_{L^2(\mathbb{T})} = 1$ for all $m\geq 3$. Since $|\widehat{\nabla u}|$ is uniformly continuous, we have that
    \begin{equation*}
        \begin{aligned}
            \|\nabla(u\ast f_m)\|_{\ell^2(\mathbb{Z})}^2 &= \frac{1}{2\pi} \int_{\mathbb{T}}|\widehat{\nabla u}(\xi)|^2 |\varphi_m(\xi)|^2 \, d\xi \\ &= \frac{1}{|B_{1/m}(\xi_0)|}\int_{B_{1/m}(\xi_0)}|\widehat{\nabla u}(\xi)|^2\, d\xi \xrightarrow[]{m \to +\infty}|\widehat{\nabla u}(\xi_0)|^2.
        \end{aligned}
    \end{equation*}
    Therefore,
    \begin{equation*}
        \|\widehat{\nabla u}\|_{L^{\infty}(\mathbb{T})} = |\widehat{\nabla u}(\xi_0)| \leq \sup_{0\neq f \in \ell^2(\mathbb{Z})} \frac{\|\nabla(u \ast f)\|_{\ell^2(\mathbb{Z})}}{\|f\|_{\ell^2(\mathbb{Z})}}.
    \end{equation*}
    This establishes \eqref{Sup1Deriv}.
    
    Next, let us consider $b_0 = u(n)$ and $b_k = u(n-k) - u(n-k+1)$ for each $k \in \{1,\dots,n\}$. Additionally, let us consider the polynomial
    \begin{equation*}
        P(z) = b_0 z^{2n+1} + b_1 z^{2n} + \dots + b_n z^{n+1} - b_n z^n - \dots - b_1 z - b_0.
    \end{equation*}
    Since
    \begin{equation*}
        \widehat{\nabla u}(\xi) = (e^{i\xi} - 1) \sum_{k=-n}^n u(k) e^{ik\xi} = e^{-in\xi}P(e^{i\xi}),
    \end{equation*}
    it follows that
    \begin{equation*}
        \sup_{0\neq f \in \ell^2(\mathbb{Z})} \frac{\|\nabla(u \ast f)\|_{\ell^2(\mathbb{Z})}}{\|f\|_{\ell^2(\mathbb{Z})}} = \|\widehat{\nabla u}\|_{L^{\infty}(\mathbb{T})} = \max_{|z|=1} |P(z)|.
    \end{equation*}
    Therefore, it suffices to show that
    \begin{equation} \label{SuffCond}
        \max_{|z|=1} |P(z)| \geq \frac{2}{2n+1}.
    \end{equation}
    Since the coefficients of $P$ are real numbers and have certain symmetry, we get that $P(z) = - z^{2n+1} \overline{P(1/\overline{z})}$. Thus, Lemma \ref{LemmaSelfInversePol} implies
    \begin{equation*}
        \max_{|z|=1} |P(z)| = \frac{2}{2n+1} \max_{|z|=1} |P'(z)|.
    \end{equation*}
    Combining the normalization and symmetry of $u$ with the definition of the $b_k$'s, we have
    \begin{equation*}
        P'(1) = (2n+1) b_0 + (2n-1)b_1 + \dots + 3b_{n-1} + b_n = 1.
    \end{equation*}
    Notice that \eqref{SuffCond} follows from the last two equations, which finishes the proof.
\end{proof}

Notice that we cannot argue exactly as before to deduce \eqref{Constant:Sharp2Restrict} since the corresponding $P'$ in such a case is not a symmetric polynomial, and thus it is not possible to apply Lemma \ref{LemmaSelfInversePol} twice. In addition, we cannot extract a suitable value of $P''$ from the normalization assumption on $u$ to carry out an analogous idea. In this case, we need to represent $\widehat{u}$ in a more convenient way by using the non-negativity hypothesis. The Fej\'er-Riesz theorem \cite[Section 1.12]{GrenanderSzego1958} is the right tool for this task. Moreover, we need to use another refined version of Bernstein's inequality for polynomials with no zeros in the open unit disk $\mathbb{D}$. This inequality will be provided by the Erd\"os-Lax theorem \cite{Lax1944}. We state these two preliminary results below and give a proof of Proposition \ref{Thm:SecondDerivative}.

\begin{lemma}[Fej\'er-Riesz]\label{Lemma:FejerRiesz}
Let $\mathbb{D} = \{z \in \mathbb{C} : |z|<1\}$. A Laurent polynomial $q(z)=\sum_{k=-m}^{m} q_k z^k$ which has complex
coefficients and satisfies $q(\zeta)\ge 0$ for all $\zeta \in \partial \mathbb{D}$ can be written
\begin{equation*}
q(\zeta)=|p(\zeta)|^2, \qquad \zeta\in \partial \mathbb{D},
\end{equation*}
for some polynomial $p(z)=p_0+p_1 z+\cdots + p_m z^m$, and $p(z)$ can be
chosen to have no zeros in $\mathbb{D}$.
\end{lemma}

\begin{lemma}[Erd\"os-Lax]\label{Lemma:ErdosLax}
    Let $P(z)$ be a polynomial of degree $n \geq 1$ with no zeros in the open unit disk $\mathbb{D}$. Then, 
\begin{equation*}
    \max_{|z|=1} |P'(z)| \leq \frac{n}{2} \max_{|z|=1} |P(z)|.
\end{equation*}
\end{lemma}

\begin{proof}[Proof of Proposition \ref{Thm:SecondDerivative}]
    Arguing as in the proof of \eqref{Sup1Deriv}, we have
    \begin{equation}\label{supLap}
        \sup_{0\neq f \in \ell^2(\mathbb{Z})} \frac{\|\nabla^2(u \ast f)\|_{\ell^2(\mathbb{Z})}}{\|f\|_{\ell^2(\mathbb{Z})}} = \|\widehat{\nabla^2 u}\|_{L^{\infty}(\mathbb{T})} = \max_{\xi \in \mathbb{T}}|(e^{i\xi}-1)^2 \widehat{u}(\xi)|.
    \end{equation}
    Since $\widehat{u}\geq 0$, an application of Fej\'er-Riesz theorem yields the existence of a polynomial
    \begin{equation*}
        Q(z) = \sum_{k=0}^n b_k z^k
    \end{equation*}
    with no zeros in $\mathbb{D}$ such that
    \begin{equation} \label{FourierDecom}
        \widehat{u}(\xi) = |Q(e^{i\xi})|^2
    \end{equation}
    for all $\xi \in \mathbb{T}$. Let us consider the polynomial $P(z) = (z-1)Q(z)$. From \eqref{supLap} and \eqref{FourierDecom}, it follows that 
    \begin{equation} \label{maxP^2}
        \sup_{0\neq f \in \ell^2(\mathbb{Z})} \frac{\|\nabla^2(u \ast f)\|_{\ell^2(\mathbb{Z})}}{\|f\|_{\ell^2(\mathbb{Z})}} = \left( \max_{|z|=1} |(z-1)Q(z)|\right)^2. 
    \end{equation}
    Notice that the normalization condition on $u$ implies that $\widehat{u}(0)=1$, and thus $|Q(1)|=1$. By applying the product rule for derivatives, we deduce that $|P'(1)|=1$. Since $P$ has no zeros in $\mathbb{D}$, Erd\"os-Lax theorem implies that
    \begin{equation*}
        1 \leq \max_{|z|=1}|P'(z)| \leq \frac{n+1}{2} \max_{|z|=1}|P(z)|.
    \end{equation*}
    Combining this inequality with \eqref{maxP^2}, we conclude
    \begin{equation*}
        \sup_{0\neq f \in \ell^2(\mathbb{Z})} \frac{\|\nabla^2(u \ast f)\|_{\ell^2(\mathbb{Z})}}{\|f\|_{\ell^2(\mathbb{Z})}} \geq \frac{4}{(n+1)^2}.
    \end{equation*}
\end{proof}

\section{Case of four derivatives with non-negative Fourier transform restriction} \label{SectionFourDerivatives}

We start this section by establishing Proposition \ref{Prop:SquarePolIneq} via an equioscillation argument, similar to that used to deduce the extremal properties of Chebyshev polynomials. Then we proceed with the proof of Theorem \ref{Thm:FourthDerivativeRest}, which can be seen as an immediate consequence of Proposition \ref{Prop:SquarePolIneq} after applying the Fourier transform.

\begin{proof}[Proof of Proposition \ref{Prop:SquarePolIneq}]
    To begin with, we shall show that $S_{n-2}$ satisfies the hypotheses of the proposition. Let us consider
    \begin{equation*}
        q(x) = (1-x)^2 S_{n-2}(x) = \frac{8}{n^2} \tan^2 \left( \frac{\pi}{2n} \right) \left( 1 + T_n\left( \frac{1+\cos(\pi/n)}{2}(x+1)-1 \right) \right).
    \end{equation*}
    Letting 
    \begin{equation*}
        \alpha =  \frac{8}{n^2} \tan^2 \left( \frac{\pi}{2n} \right) \quad \text{and} \quad L(x) = \frac{1+\cos(\pi/n)}{2}(x+1)-1,
    \end{equation*}
    we have $q(x) = \alpha (1 + T_n(L(x)))$. This implies that $q$ is a non-negative polynomial of degree $n$. Since $T_n(\cos(\pi/n))=-1$ and $T_n'(\cos(\pi/n))=0$, it follows that $q(1)=q'(1)=0$. This means that $(1-x)^2$ is a factor of $q$, and thus $S_{n-2}$ is a non-negative polynomial on $[-1,1]$ of degree $n-2$. Moreover, 
    \begin{equation*}
        S_{n-2}(1) = \frac{1}{2}q''(1) = \frac{\alpha}{2} (L'(1))^2 T_n''(\cos(\pi/n)) = \frac{\alpha}{2} \left( \frac{1+\cos(\pi/n)}{2} \frac{n}{\sin(\pi/n)} \right)^2 = 1.
    \end{equation*}
    Notice that equality in \eqref{PolynomIneq} is attained for $S_{n-2}$, because
    \begin{equation*}
        \max_{x \in [-1,1]}|(1-x)^2 S_{n-2}(x)| = \max_{x \in [-1,1]} |q(x)| = 2 \alpha = \frac{16}{n^2} \tan^2 \left( \frac{\pi}{2n} \right).
    \end{equation*}
    Next, let $p$ be a polynomial of degree at most $n-2$ that is non-negative on $[-1,1]$ such that $p(1)=1$ and
    \begin{equation} \label{HypP}
        \max_{x \in [-1,1]}|(1-x)^2 p (x)| \leq \frac{16}{n^2} \tan^2 \left( \frac{\pi}{2n} \right).
    \end{equation}
    We shall prove that $p = S_{n-2}$. Since $p(1) = S_{n-2}(1) = 1$, there exists a polynomial $r$ of degree at most $n-3$ such that 
    \begin{equation*}
        p(x) - S_{n-2}(x) = (1-x)r(x).
    \end{equation*}
    Hence,
    \begin{equation}\label{SignR}
        (1-x)^2p(x) - q(x) = (1-x)^3 r(x).
    \end{equation}
    If $y \in [-1,1)$ is such that $q(y) = 0$, then the non-negativity of $p$ on $[-1,1]$ and \eqref{SignR} imply that $r(y) \geq 0$. If $z \in [-1,1)$ is such that $q(z)=2\alpha$, then \eqref{HypP} and \eqref{SignR} yield that $r(z) \leq 0$. Observe that there are $n-1$ interlaced points in the interval $[-1,1)$ satisfying exactly one of the above two conditions, say $-1 = x_1 < x_2 < \dots < x_{n-1} < 1$. By the intermediate value theorem, there exist $y_1, \dots, y_{n-2} \in [-1,1)$ such that $y_i \in [x_i,x_{i+1}]$ and $r(y_i) = 0$ for each $i \in \{1,\dots,n-2\}$. If $y_{i} = x_{i+1} = y_{i+1} $ for some $i \in \{1,\dots,n-2\}$, the fact that $0 \leq (1-x)^2 p(x) \leq 2 \alpha$ in $(-1,1)$ implies that $y_i$ is a critical point of $(1-x)^2p(x)$. Taking derivative in \eqref{SignR}, we deduce that $y_i$ is a critical point of $(1-x)^3r(x)$. Since $r(y_i) = 0$ and $y_{i} = x_{i+1} \neq 1$, it follows that $r'(y_i) = 0$. This means that $y_i$ is a double root of $r$ whenever $y_i = x_{i+1} = y_{i+1}$. As a result, $r$ has $n-2$ roots in $[-1,1)$ counting multiplicities. Hence, $r \equiv 0$. This finishes the proof.
\end{proof}

\begin{proof}[Proof of Theorem \ref{Thm:FourthDerivativeRest}]
    Let $u : \{-n,\dots,n \} \to \mathbb{R}$ be a symmetric kernel with non-negative Fourier transform such that $\sum_{k=-n}^{n}u(k) = 1$. Because of the symmetry of $u$, we get
    \begin{equation*}
        \widehat{u}(\xi) = \sum_{k=-n}^{n} u(k) e^{-ik\xi} = u(0) + 2 \sum_{k=1}^{n}u(k) \cos(k\xi) = u(0) + \sum_{k=1}^{n} 2u(k) T_k(\cos(\xi)),
    \end{equation*}
    where $T_k$ denotes the Chebyshev polynomial of degree $k$. By applying Fourier transform and Plancherel's theorem as in \eqref{Sup1Deriv} together with the change of variable $x = \cos(\xi)$, we deduce
    \begin{equation} \label{SupFourDerivatives}
        \sup_{0\neq f \in \ell^2(\mathbb{Z})} \frac{\|\nabla^4(u\ast f)\|_{\ell^2(\mathbb{Z})}}{\|f\|_{\ell^2(\mathbb{Z})}} = \| (e^{i\xi}-1)^4 \widehat{u}(\xi)\|_{L^{\infty}(\mathbb{T})} = 4 \|(1-x)^2 p_u(x)\|_{L^{\infty}([-1,1])},
    \end{equation}
    where 
    \begin{equation}\label{PolV}
        p_u(x) := u(0) + \sum_{k=1}^{n} 2u(k) T_k(x).
    \end{equation}
    Since $p_u$ is a polynomial of degree at most $n$ that is non-negative on $[-1,1]$, Proposition \ref{Prop:SquarePolIneq} yields 
    \begin{equation}\label{Max(1-x)^2p}
        \|(1-x)^2 p_u(x)\|_{L^{\infty}([-1,1])} \geq \frac{16}{(n+2)^2} \tan^2\left( \frac{\pi}{2n + 4} \right).
    \end{equation}
    Combining \eqref{SupFourDerivatives} and \eqref{Max(1-x)^2p}, we obtain
    \begin{equation}\label{Sharp4DerivNew}
        \sup_{0\neq f \in \ell^2(\mathbb{Z})} \frac{\|\nabla^4(u\ast f)\|_{\ell^2(\mathbb{Z})}}{\|f\|_{\ell^2(\mathbb{Z})}} \geq \frac{64}{(n+2)^2} \tan^2\left( \frac{\pi}{2n + 4} \right).
    \end{equation}
    We shall prove that equality in \eqref{Sharp4DerivNew} is only attained for the kernel defined in \eqref{OptimalKernel4Deriv}.
    
    Let $S_n$ be the extremal polynomial of Proposition \ref{Prop:SquarePolIneq}, which was defined in \eqref{PolynomS_n}. Since the set of Chebyshev polynomials $\{T_0=1,T_1,\dots,T_n\}$ is a basis of the vector space of real-valued polynomials of degree at most $n$, there exist unique constants $\alpha_0,\dots,\alpha_n \in \mathbb{R}$ such that
    \begin{equation} \label{RepresentationSn}
        S_n(x) = \alpha_0 T_0(x) + \sum_{k=1}^n 2\alpha_k T_k(x).
    \end{equation}
    Since $S_n(1)=1$ and $T_k(1)=1$ for every $k \in \{0,\dots,n\}$, letting $x=1$ in the above equation we get 
    \begin{equation}\label{NormalizedV}
        \alpha_0 +  \sum_{k=1}^n 2\alpha_k = 1.
    \end{equation}
    In addition, we know that Chebyshev polynomials satisfy the orthogonality property
    \begin{equation*}
        \int_{-1}^{1}T_{m_1}(x)T_{m_2}(x) \frac{dx}{\sqrt{1-x^2}} = \begin{cases}
            0, &\text{if } m_1 \neq m_2,\\
            \pi, &\text{if } m_1 = m_2 = 0,\\
            \pi/2, &\text{if } m_1 = m_2 \neq 0.
        \end{cases}
    \end{equation*}
    Applying this orthogonality in \eqref{RepresentationSn}, we have
    \begin{equation*}
        \alpha_k = \frac{1}{\pi}\int_{-1}^{1}S_n(x)T_{k}(x)\frac{dx}{\sqrt{1-x^2}}
    \end{equation*}
    for each $k \in \{0,\dots,n\}$.
    
    Let us consider the symmetric kernel $v:\{-n,\dots,n\} \to \mathbb{R}$ defined in \eqref{OptimalKernel4Deriv}. It follows that $v(k) = v(-k) = \alpha_k$ for all $k \in \{0,\dots,n\}$, and thus $v$ is a normalized symmetric kernel by \eqref{NormalizedV}. Moreover, the polynomial $p_v$ (as defined in \eqref{PolV}) associated to $v$ is given by 
    \begin{equation*}
        p_v(x) = v(0) + \sum_{k=1}^{n} 2v(k) T_k(x) = \alpha_0 T_0(x) + \sum_{k=1}^n 2\alpha_k T_k(x) = S_n(x).
    \end{equation*}
    This implies that $v$ has non-negative Fourier transform. Additionally, by Proposition \ref{Prop:SquarePolIneq}, we deduce
    \begin{equation*}
        \sup_{0\neq f \in \ell^2(\mathbb{Z})} \frac{\|\nabla^4(v\ast f)\|_{\ell^2(\mathbb{Z})}}{\|f\|_{\ell^2(\mathbb{Z})}} = 4 \|(1-x)^2 S_n(x)\|_{L^{\infty}([-1,1])} = \frac{64}{(n+2)^2} \tan^2\left( \frac{\pi}{2n + 4} \right).
    \end{equation*}
    This shows that equality in \eqref{Sharp4DerivNew} is attained by the kernel defined in \eqref{OptimalKernel4Deriv}, which proves that such an inequality is optimal.

    Now, let $w:\{-n,\dots,n\} \to \mathbb{R}$ be a normalized symmetric kernel  with non-negative Fourier transform that attains equality in \eqref{Sharp4DerivNew}. This implies that
    \begin{equation*}
        \|(1-x)^2 p_w(x)\|_{L^{\infty}([-1,1])} = \frac{16}{(n+2)^2} \tan^2\left( \frac{\pi}{2n + 4} \right).
    \end{equation*}
    By Proposition \ref{Prop:SquarePolIneq}, we deduce that $p_w(x) = S_n(x)$. That is,
    \begin{equation*}
        w(0) + \sum_{k=1}^{n} 2w(k) T_k(x) = S_n(x).
    \end{equation*}
    Applying once again the orthogonality property of Chebyshev polynomials, we get  
    \begin{equation*}
        w(k) = \frac{1}{\pi}\int_{-1}^{1}S_n(x)T_{k}(x)\frac{dx}{\sqrt{1-x^2}} 
    \end{equation*}
    for every $k \in \{0,\dots,n\}$. This shows that there is no normalized symmetric kernel with non-negative Fourier transform different from that of \eqref{OptimalKernel4Deriv} for which equality is attained in \eqref{Sharp4DerivNew}.
\end{proof}

\section{Case of six derivatives with non-negative Fourier transform restriction} \label{SectionSixDerivatives}

\subsection{Preliminaries: Construction and properties of the Zolotarev polynomials of the first kind and second type}
In several works spanning from 1868 to 1877, Zolotarev studied the relationship between elliptic functions and functions of a complex variable, developing an algorithm to identify whether an integral of the form 
\begin{equation}\label{Integral Elliptic}
  \int\dfrac{A-x}{\sqrt{x^4+ax^3+bx^2+cx+d}} \ dx 
\end{equation}
was elementary, in this case represented in terms of logarithms. Zolotarev's method improved previous works of Abel by obtaining an algorithm with real practical use. This led to developing a method of constructing two families of polynomials, expressed explicitly in terms of elliptic functions, with an alternating behavior and with applications to least deviation from zero problems. Lebedev \cite{Levedev1994} studied the properties of the Zolotarev polynomials of the first kind and second type (among other types) and applied them to solve some extremal problems in which the set of polynomials is constrained by some additional requirements. One of the problems solved by Lebedev is related to Proposition \ref{(1-x)^3 theorem}. Here, we will summarize the method of constructing this family of polynomials presented in \cite[Section 1 and 2]{Levedev1994}, specially for the case of Zolotarev polynomials of the first kind and second type, and will go through the proof of Proposition \ref{(1-x)^3 theorem}, which is our main interest in this paper.

Let $\Pi_n$ be the class of monic polynomials of degree $n$. We aim to find the extremal polynomial that solves the problem $$\inf_{Q_n\in\Pi_n} \max_{x\in[-1,1]}|Q_n(x)\rho(x)|$$
for some weight function of the form $\rho(x)=\left(R_M(x)/S_L(x)\right)^{1/2}$, where $R_M$ and $S_L$ are polynomials of degrees $M$ and $L$, respectively, and $M$ and $L$ satisfy $2n + M > L$. Let $P_n \in \Pi_n$, and set 
\begin{equation}\label{set F_n}
F_n(x)=P_n(x)\rho(x)  
\end{equation}
be such that there exist real values $x_1> x_2 > \dots >x_l$ for which the equalities 
\begin{equation}\label{EqFnLn}
    \begin{aligned}
        F_n^2(x)-L_n^2=0\\
    F'_n(x)=0
    \end{aligned}
\end{equation}
hold simultaneously for some $L_n>0$ at each $x=x_i$, and $F_n^2(x)\leq L_n^2$ for all $x\in[-1,1]$. If we let $\displaystyle \varphi_l(x)=\Pi_{i=1}^l(x-x_i)$, then \eqref{EqFnLn} can be rewritten as
\begin{equation}\label{Fn2-Ln2}
  F_n^2(x)-G(x)\varphi_l^2(x)=L_n^2 
\end{equation}
for some function $G(x)$ that will be mentioned later and satisfies $G(x) \leq 0$ for all $x \in [-1,1]$. This is called the \textit{Fermat-Abel functional equation}. A method for finding extremal polynomials was developed by Chebyshev, Lebedev presents a generalized form \cite[p. 233]{Levedev1994}, and we expose this method here. Observe that for the derivative $F_n'$, the following equality is valid by (\ref{set F_n}) and $(\ref{Fn2-Ln2})$
\begin{equation}\label{derivative F_n equation}
    F_n'(x)=\dfrac{ t \cdot  Q_{\bar{p}}(x)\varphi_l(x)}{R_M^{1/2}(x)S_L^{3/2}(x)}
\end{equation}
where $t$ is some constant and
\begin{equation*}
    t \cdot  Q_{\bar{p}}(x)\varphi_l(x)=R_M(x) S_L(x) P_n'(x) + \frac{1}{2} \left( \dfrac{R_M}{S_L}\right)'(x) \cdot S_L^2(x) P_n(x).
\end{equation*}
Then (\ref{Fn2-Ln2}) and (\ref{derivative F_n equation}) give us the so-called \textit{Chebyshev differential equation}
\begin{equation}\label{Chebyshev's diff equations}
    \frac{F_n'}{\sqrt{L_n^2-F_n^2}}= \frac{t \cdot  Q_{\bar{p}}(x)}{S_L(x)\sqrt{-G(x)}} \quad \text{or} \quad \frac{F_n'}{\sqrt{F_n^2-L_n^2}}= \frac{t \cdot  Q_{\bar{p}}(x)}{S_L(x)\sqrt{G(x)}}.
\end{equation}
It is considered that $\sqrt{G(x)}$ corresponds to the principal branch of the complex square root function. Analytical solutions to the Chebyshev's differential equation can be derived from the function 
\begin{equation}\label{Elliptic integral G}
    I(x,x_0)=-\int_{x_0}^x \dfrac{s\cdot Q_{\bar{p}}(x)}{S_L(x)\sqrt{-G(x)}} \, dx,
\end{equation}
where $s=t/r$ and $r$ is the coefficient of the leading term of $R_M(x)$. For the particular case of Zolotarev polynomials of the first kind, we choose $S_L(x)\equiv 1$ and $G(x)$ to be a quartic polynomial that has the form
\begin{equation}\label{weight function}
    G(x)=(x^2-1)(x-\alpha)(x-\beta) \ \text{ or } \ G(x)=(x^2-1)((x-\gamma)^2+\varepsilon^2),
\end{equation}
see \cite[Section 2 and the discussion at the end of p. 239]{Levedev1994}. 

The first option implies that $G(x)$ has four different real roots and gives rise to the Zolotarev polynomials of the first kind and the first type, whereas the second option is when  $G(x)$ has two real roots and two complex conjugate roots, this one gives us the Zolotarev polynomials of the first kind and the second type, which is the family that we are interested in. Thus we proceed by showing the construction of such polynomials as in \cite[p. 249]{Levedev1994}. From now on, we consider $G(x)=(x^2-1)((x-\gamma)^2+\varepsilon^2)$ and $Q_{\bar{p}}(x) = x-c$. Therefore, \eqref{Chebyshev's diff equations} and \eqref{Elliptic integral G} become
\begin{equation}\label{arccos int}
    \frac{P_n'}{\sqrt{L_n^2-P_n^2}}= \frac{n(x-c)}{\sqrt{-G(x)}}
\end{equation}
and
\begin{equation}\label{Elliptic integral expanded}
    I(x,x_0)=\int_{x_0}^x \dfrac{x-c}{\sqrt{(1-x^2)((x-\gamma)^2+\varepsilon^2)}} \, dx.
\end{equation} 
Notice that the left-hand side of \eqref{arccos int} equals $\dfrac{d}{dx}\left(\arccos\left(\dfrac{P_n(x)}{L_n}\right)\right)$, then by using the antiderivative \eqref{Elliptic integral expanded} and the fact that $P_n(1)=L_n$ (which follows from \eqref{Fn2-Ln2}), we conclude
\begin{equation}\label{arccos of Z2_n}
  \arccos\left(\dfrac{P_n(x)}{L_n}\right)=nI(x,1).
\end{equation}
Let $1\leq \mu\leq n/2$. One can get the value $nI(1,-1)=(n-2\mu)\pi$ by setting the parameters
\begin{equation*}
a=\dfrac{\mu K(k)}{n},    
\end{equation*}
\begin{equation}\label{gamma and epsilon}
 \gamma=\dfrac{\operatorname{cn}(2a)}{\operatorname{dn}^2(2a)}, \quad \varepsilon=\dfrac{kk'\operatorname{sn}^2(2a)}{\operatorname{dn}^2(2a)}, 
\end{equation}
and 
\begin{equation}\label{defC}
 c=1+\dfrac{1}{\operatorname{dn}(2a)}   (\operatorname{cn}(2a)-1+2\operatorname{sn}(2a)Z(a)).
\end{equation}
Moreover, the solution to the Chebyshev's differential equation \eqref{arccos int} exists and is the monic polynomial
\begin{equation*}
    P_n(x)=L_nZ2_n(x,k,\mu)
\end{equation*}
where
\begin{equation*}
    L_n = \dfrac{1}{2^{n-1}}\left( \dfrac{\Theta_3(0)}{\Theta_4(2a_n)} \right),
\end{equation*}
\begin{equation*}
Z2_n(x,k,\mu)=\dfrac{1}{2}\left[\left(\dfrac{\Theta_1(a+u) \Theta_3(a+u) }{\Theta_1(a-u) \Theta_3(a-u)}\right)^n+\left(\dfrac{\Theta_1(a-u) \Theta_3(a-u)}{\Theta_1(a+u) \Theta_3(a+u) }\right)^n \ \right],   
\end{equation*}
and the variable is
\begin{equation*}\label{definition variable x}
    x=\dfrac{\operatorname{cn}(2a)\operatorname{cn}(2u)-1}{\operatorname{cn}(2u)-\operatorname{cn}(2a)},
\end{equation*}
with $u$ ranging over $[0,K(k')i]$ when $x$ ranges over $[-1,1]$. Here, $k'$ denotes the \textit{complementary modulus} $k'=\sqrt{1-k^2}$. 
The polynomial $Z2_n(x,k,\mu)$ of degree $n$ is called the \textit{Zolotarev polynomial of the
first kind and second type}. 

It is convenient to take $c, k, u, \mu$ and $a$  as parameters in extremal problems for the Zolotarev polynomials. The geometric interpretation of the parameters is as follows: For $c<1$, the polynomial $Z2_n(x,k,\mu)$ has $n-2\mu+1$ points of equioscillation taking the values $-1$ and $1$ on the interval $[-1,c)$, and $2\mu-1$ of these points on $(c,1]$. The value $c=1$ is the extreme case, which can only occur when $\mu=1$, and in such case we have exactly $n-1$ points of equioscillation between $-L_n$ and $L_n$ on the interval $[-1,1]$, see Figure \ref{fig:2}.

Similarly, the following 
\begin{equation*}
 V2_{n-2}(x,k,\mu):=\dfrac{1}{2\sqrt{G(x)}}\left[\left(\dfrac{\Theta_1(a+u) \Theta_3(a+u) }{\Theta_1(a-u) \Theta_3(a-u)}\right)^n-\left(\dfrac{\Theta_1(a-u) \Theta_3(a-u)}{\Theta_1(a+u) \Theta_3(a+u) }\right)^n \ \right]
\end{equation*}
defines the polynomial $V2_{n-2}(x,k,\mu)$ of degree $n-2$ that is called the \textit{Zolotarev polynomial of the
second kind and second type}.

Next, we mention some properties of the polynomials $Z2_n(x,k,\mu)$ and $V2_{n}(x,k,\mu)$ that will be useful to solve the extremal problem from Proposition \ref{(1-x)^3 theorem}. The proof of these properties may be found in \cite[Lemma 2.8]{Levedev1994}.

\begin{lemma}\label{properties Z2_n}
 If $\mu$ satisfies $1\leq\mu \leq\frac{n}{2}$, then
\begin{equation}\label{MaxZ2n}
 \max_{x\in[-1,1]}|Z2_n(x,k,\mu)|=1  ,
\end{equation}
 \begin{equation*}
 \max_{x\in[-1,1]*}|\sqrt{-G(x)}V2_{n-2}(x,k,\mu)|=1,   
\end{equation*}
 \begin{equation}\label{second derivative is zero when c=1}
Z2_n(x,k,\mu)^2-G(x)V2_{n-2}(x,k,\mu)^2=1   ,
\end{equation}
 \begin{equation}\label{first derivative is zero at c}
Z2_n'(x,k,\mu)=n(x-c)V2_{n-2}(x,k,\mu).   
\end{equation}
 
\end{lemma}
\subsection{A sharp polynomial inequality}
Before proving Proposition \ref{(1-x)^3 theorem}, we discuss the following problem \cite[Problem BZ21]{Levedev1994} and its solution.

\begin{problem}\label{ProblemBZ21}
    Among all the polynomials $Q_n\in \Pi_n$ satisfying 
\begin{equation*}
    Q_n'(1)=Q_n''(1)=0,
\end{equation*}
find the polynomial of least deviation from zero on $[-1,1]$ (i.e. the polynomial that minimizes $\max_{x \in [-1,1]}|Q_n(x)|$).
\end{problem}

The solution to this problem is the monic polynomial $L_nZ2_n(x,k_n,1)$ whose modulus $k_n$ satisfies $c(1,k_n)=1$, where $c$ is defined as in \eqref{defC}. As discussed above, the value $c=1$ is the extreme case, which can only exist when $\mu=1$. In such a case, we have exactly $n-1$ equioscillation values, all of them on the interval $[-1,1]$. See Figure \ref{fig:2} for the case $n=10$.

This assertion can be verified more formally by using the properties of $Z2_n(x,k,\mu)$ coming from its definition and Lemma \ref{properties Z2_n}. This is done by Lebedev, but we include the details here for completeness.
\begin{proof}[Solution of Problem \ref{ProblemBZ21}]
We first verify that $Z2_n'(1,k_n,1)=Z2_n''(1,k_n,1)=0$. The parameters are chosen to be $\mu=1$ and  $k_n$ to be the solution to $c(1,k_n)=1$. Hence, Lemma \ref{properties Z2_n} applies and \eqref{first derivative is zero at c} takes the form
\begin{equation}\label{first derivative}
 Z2'_n(x,k_n,1)=n(x-1)V2_{n-2}(x,k_n,1).    
\end{equation}
This implies that $Z2'_n(1,k_n,1)=0$. Next, by differentiating (\ref{second derivative is zero when c=1}) we get
\begin{equation}\label{V_2Vanish}
    2Z2_{n}\cdot Z2_{n}'-G'\cdot V2_{n-2}^2-2G\cdot V2_{n-2}\cdot V2_{n-2}'=0.
\end{equation}
Recall that $G(x)=(1-x^2)((x-\gamma)^2+\varepsilon^2)$, which means that $G(1)=0$ and $G'(1)\neq 0$. Observe that \eqref{first derivative} yields that $Z2_n'(1)=0$. Combining these equations and \eqref{V_2Vanish}, we deduce that $V2_{n-2}(1)=0$. From \eqref{first derivative}, it follows that $Z2'_n(x,k_n,1)$ has a double root at $x=1$, and thus $Z2''_n(1,k_n,1)=0$.

The least deviation assertion follows from the equioscillation property of $Z2_n(x,k,1)$ on this interval. Indeed, by expanding (\ref{arccos of Z2_n}) we notice that $$\arccos(Z2_n(x,k_n,1))=nI(x,1)=\int_{x}^1 \dfrac{n(1-x)}{\sqrt{(1-x^2)((x-\gamma)^2+\varepsilon^2)}}\, dx.$$ Since the integrand is purely positive in $[-1,1)$, it follows that $\arccos(Z2_n(x,k_n,1))$ is continuous and monotonically decreasing on $[-1,1]$. Moreover, when $\mu=1$, we have the condition $nI(-1,1)=(n-2)\pi$. Given that $I(1,1)=0$, the intermediate value theorem yields that $\arccos (Z2_n(x,k_n,1))$ takes each of the values $0,\pi,2\pi,\dots,(n-2)\pi$ exactly once in $[-1,1]$. As a result, there exist exactly $n-1$ points in $[-1,1]$ for which $Z2_n(x,k_n,1)$ takes the values $\pm 1$ (which are local extrema of this polynomial by \eqref{MaxZ2n}), and these points are interlaced. This establishes the equioscillation property of $Z2_n(x,k_n,1)$ in the interval $[-1,1]$ when $k_n$ is chosen to solve $c(1,k_n)=1$. In order to deduce that $Z2_n(x,k_n,1)$ solves Problem \ref{ProblemBZ21}, we have to use the equioscillation property and argue as in the proof of Proposition \ref{Prop:SquarePolIneq}. We omit the details to avoid repetition.
\end{proof}
The polynomial required to prove Proposition \ref{(1-x)^3 theorem} is a modification of $Z2_n(x,k_n,1)$. More precisely, we need 
\begin{equation}\label{poly Z2_n modified}
\bar{Z2}_n(x,k_n,1) := \dfrac{9}{n^2}\left(\frac{1-\operatorname{cn}(2a_n)}{\operatorname{dn}(2a_n)}\right)^2\dfrac{1}{(1-x)^3} \ (1-Z2_{n}(x,k_n,1)).
\end{equation}
In the next proposition, we state the properties of $\bar{Z2}_n(x,k_n,1)$ that allow us to see that this is the extremal polynomial for \eqref{PolynomIneq6}.
\begin{proposition}\label{Prop: Zolotarev modified properties} The following properties hold:
\begin{enumerate}
    \item $0\leq1-Z2_{n}(x,k_n,1)\leq 2$ has the $n-1$ equioscillation property on $[-1,1]$ that alternates from $0$ to $2$.
    \item The function $\bar{Z2}_n(x,k_n,1)$ as defined in $(\ref{poly Z2_n modified})$ is a polynomial of degree $n-3$.
    \item The polynomial $\bar{Z2}_n(x,k_n,1)$ satisfies $\bar{Z2}_n(1)=1$.    
    \end{enumerate}
    \end{proposition}
    \begin{proof}
     Part (1) follows immediately from the equioscillation property of $Z2_{n}(x,k_n,1)$ between the values $-1$ and $1$ on the interval $[-1,1]$.  For part (2), we plug $x=1$ in \eqref{arccos of Z2_n} to get $$\arccos(Z2_n(1,k_n,1))=nI(1,1)=0\ \  \text{ or }\ \ Z2_n(1,k_n,1)=1.$$ Since $Z2_n'(1,k_n,1) = Z2_n''(1,k_n,1) = 0$, it follows that $1-Z2_{n}(x,k_n,1)$ has a triple root at $x=1$. As a result, $\bar{Z2}_n(x,k_n,1)$ is a polynomial of degree $n-3$

     Part (3) is more interesting. This problem is equivalent to finding the third coefficient of the Taylor's expansion of $1-Z2_{n}(x,k_n,1)$ at $x=1$. Thus, our goal is to calculate $z_3=\frac{Z2_{n}'''(1,k_n,1)}{3!}$.  Let us consider the Taylor's expansions
\begin{equation*}
 Z2_{n}(x,k_n,1)=1+z_3(x-1)^3+\mathcal{O}((x-1)^4)   
\end{equation*}
and
\begin{equation*}
V2_{n-2}(x,k_n,1)=v_1(x-1)+\mathcal{O}((x-1)^2),    
\end{equation*}
where $v_1=V2_{n-2}'(1,k_n,1)$. It follows that
$$Z2_{n}(x,k_n,1)^2=1+2z_3(x-1)^3+\mathcal{O}((x-1)^4)$$ and $$V2_{n}(x,k_n,1)^2=v_1^2(x-1)^2+\mathcal{O}((x-1)^3).$$  Since $G(x)$ has a simple root at $x=1$, we can take the third Taylor coefficient on both sides of \eqref{second derivative is zero when c=1} to deduce
\begin{equation}\label{Z3V1-1}
    2z_3-G'(1)v_1^2=0.
\end{equation}
Now, by differentiating \eqref{first derivative} twice an plugging $x=1$ in the resulting equation, we obtain
$$Z2_n'''(1,k_n,1)=2nV2_{n-2}'(1,k_n,1)=2nv_1,$$
which implies that
\begin{equation}\label{Z3V1-2}
    z_3=\frac{2n}{3!}v_1.
\end{equation}
Combining \eqref{Z3V1-1} and \eqref{Z3V1-2}, we get $$\frac{4n}{3!}v_1=G'(1)v_1^2 \ \ \text{ or }\ \ \frac{4n}{3!}\frac{1}{G'(1)}=v_1  .$$ Therefore, $z_3=\frac{8n^2}{(3!)^2}\frac{1}{G'(1)}$. Finally, by the choices of $G(x)$ and the parameters $\gamma$ and $\varepsilon$ in \eqref{gamma and epsilon}, we have
\begin{align*}
   G'(1)&=2((1-\gamma)^2+\varepsilon^2)\\
   &=2\left(\left(1-\frac{\operatorname{cn}(2a_n)}{\operatorname{dn}^2(2a_n)}\right)^2+\left(\frac{kk'\operatorname{sn}^2(2a_n)}{\operatorname{dn}^2(2a_n)}\right)^2\right)\\
    &=\frac{2}{\operatorname{dn}^4(2a_n)}\left(\left(\operatorname{dn}^2(2a_n)-\operatorname{cn}(2a_n)\right)^2+\left(kk'\operatorname{sn}^2(2a_n)\right)^2\right).
\end{align*}
Using the identities $k^2\operatorname{sn}^2(u)+\operatorname{dn}^2(u)=1$ and $(k')^2\operatorname{sn}^2(u)+\operatorname{cn}^2(u)=\operatorname{dn}^2(u)$, we get 
\begin{align*}
 \left(kk'\operatorname{sn}^2(2a_n)\right)^2&=\Big(1-\operatorname{dn}^2(2a_n)\Big)\cdot\Big(\operatorname{dn}^2(2a_n)-\operatorname{cn}^2(2a_n)\Big)\\   
 &=\operatorname{dn}^2(2a_n)-\operatorname{cn}^2(2a_n)-\operatorname{dn}^4(2a_n)+\operatorname{dn}^2(2a_n)\operatorname{cn}^2(2a_n).
\end{align*}
Notice that the terms $-\operatorname{cn}^2(2a_n)$ and $-\operatorname{dn}^4(2a_n)$ appear in the expansion of $\left(\operatorname{dn}^2(2a_n)+\operatorname{cn}(2a_n)\right)^2$ with opposite sign. After cancellation, we get
\begin{align*}
   G'(1)&=\frac{2}{\operatorname{dn}^4(2a_n)}\left(\left(\operatorname{dn}^2(2a_n)-\operatorname{cn}(2a_n)\right)^2+\left(kk'\operatorname{sn}^2(2a_n)\right)^2\right)\\
   &=\frac{2}{\operatorname{dn}^4(2a_n)}\Big(-2\operatorname{dn}^2(2a_n)\operatorname{cn}(2a_n)+\operatorname{dn}^2(2a_n)+\operatorname{cn}^2(2a_n)\operatorname{dn}^2(2a_n)\Big)\\
   &=\frac{2}{\operatorname{dn}^4(2a_n)}\Big(\operatorname{dn}^2(2a_n)\Big( 1-2\operatorname{cn}(2a_n)+\operatorname{cn}^2(2a_n)\Big)\Big)\\
    &=2\left(\frac{1-\operatorname{cn}(2a_n)}{\operatorname{dn}(2a_n)}\right)^2.
\end{align*}
Hence, $$\displaystyle z_3=\dfrac{8n^2}{(3!)^2} \frac{1}{G'(1)} =\frac{4n^2}{(3!)^2}\left(\dfrac{\operatorname{dn}(2a_n)}{1-\operatorname{cn}(2a_n)}\right)^2=\frac{n^2}{9}\left(\frac{\operatorname{dn}(2a_n)}{1-\operatorname{cn}(2a_n)}\right)^2.$$ 
By \eqref{poly Z2_n modified}, we conclude
\begin{equation*}
    \bar{Z2}_n(1,k_n,1) = \dfrac{9}{n^2}\left(\frac{1-\operatorname{cn}(2a_n)}{\operatorname{dn}(2a_n)}\right)^2 \lim_{x\to 1} \frac{-z_3(x-1)^3 + \mathcal{O}((x-1)^4)}{(1-x)^3} = 1.
\end{equation*}
This finishes the proof.
    \end{proof}

\begin{proof}[Proof of Proposition \ref{(1-x)^3 theorem}]
By Proposition \ref{Prop: Zolotarev modified properties}, we know that $\bar{Z2}_n(1,k_n,1)=1$ and that the $n$-th degree polynomial $(1-x)^3\bar{Z2}_n(x,k_n,1)$ has the equioscillation property with amplitude $\frac{18}{n^2}\left(\frac{1-\operatorname{cn}(2a_n)}{\operatorname{dn}(2a_n)}\right)^2$ for exactly $n-2$ points on the interval $[-1,1)$ (by part (1) of Proposition \ref{Prop: Zolotarev modified properties} , this polynomial has the property for $n-1$ points in $[-1,1]$).  Thus, an application of the same equioscillation argument as in the proof of Proposition \ref{Prop:SquarePolIneq} establishes Proposition \ref{(1-x)^3 theorem}.
\end{proof}

To conclude this section, we study the asymptotic behavior of the optimal constant in Theorem \ref{Thm:SixthDerivativeRest} for large values of $n$, which was stated in Remark \ref{Remark asymptotycs}. In Table \ref{tab:ValuesKn}, we compute the numerical values of $k_n$ for $3\leq n\leq 80$, where $k_n$ is the solution to \eqref{c=1}.

\begin{proof}[Proof of Remark \ref{Remark asymptotycs}]

\begin{table}[h]
\centering
\begin{tabular}{|c| c|}
\hline
$n$ &
$k_n$ \\
\hline
$3$  & $0.9596390$ \\
$4$  & $0.9332251$ \\
$5$  & $0.9236033$ \\
$6$  & $0.9188118$ \\
$7$  & $0.9160551$ \\
$8$  & $0.9143169$ \\
\hline
\end{tabular}
\hspace{0.2cm}\dots \dots\hspace{0.2cm}
\begin{tabular}{|c| c|}
\hline
$n$ &
$k_n$ \\
\hline
$75$  & $0.9089678$ \\
$76$  & $0.9089662$ \\
$77$  & $0.9089647$ \\
$78$  & $0.9089633$ \\
$79$  & $0.9089619$ \\
$80$  & $0.9089606$ \\
\hline
\end{tabular}
\caption{Numerical values of some $k_n$'s from \eqref{c=1}.}
\label{tab:ValuesKn}
\end{table}

We want to find the asymptotic solution to $c(k_n,1)=1$ as $n\to \infty$. Recall that $c$ is given by  the equation (\ref{defC}), and thus the equation $c=1$ is equivalent to 
\begin{equation}\label{c=1New}
    \operatorname{cn}\left(\frac{2K(k_n)}{n}\right)+2\operatorname{sn}\left(\frac{2K(k_n)}{n}\right)Z\left(\frac{K(k_n)}{n}\right)=1.
\end{equation}
    Since the sequence $\{k_n\}$ can be bounded by a number strictly less than $1$, it follows that $K(k_n)$ is bounded, which implies that  $a_n=\dfrac{K(k_n)}{n}\to 0$ as $n\to \infty$. Using the asymptotics of the elliptic functions around $a\sim 0$ \cite[\href{https://dlmf.nist.gov/22.10}{Jacobi E.F.}]{NIST:DLMF}, we have $\operatorname{cn}(2a)\sim 1-2a^2$, $\operatorname{sn}(2a)\sim 2a$,  $\operatorname{dn}(2a)\sim 1-2k^2a^2$ and $Z(a)\sim (1-E/K)a$. Plugging the asymptotic approximations in \eqref{c=1New} as $n\to \infty$, we obtain 
\begin{align*}
    (1-2a^2)+2(2a)((1-E/K)a)\sim 1.
\end{align*}
That is,
\begin{equation*}
    -2a^2+4a^2(1-E(k)/K(k))\sim 0.
\end{equation*}
It follows that $(1-E(k)/K(k))\sim 1/2$, which implies that $ K(k) \sim 2E(k)$. The numerical solution to $$K(k^*) = 2E(k^*)$$ is $k^*\approx 0.9089$ that gives $K(k^*)\approx2.3210$ and then $a_n \approx \frac{2.3210}{n}$ for large $n$. Hence, the asymptotic value of the right-hand side of (\ref{Sharp6Deriv}) is  $$\left(\frac{1-\operatorname{cn}(2a_{n+3})}{\operatorname{dn}(2a_{n+3})}\right)^2\sim\left(\frac{2a_{n+3}^2}{1-2k_{n+3}^2a_{n+3}^2}\right)^2\sim 4a_{n+3}^4\sim\frac{4K(k^*)^4}{(n+3)^4}$$ when $a\sim 0$, or equivalently,  when $n\to +\infty$.

\begin{figure}[h]
    \centering
\includegraphics[width=0.75\linewidth]{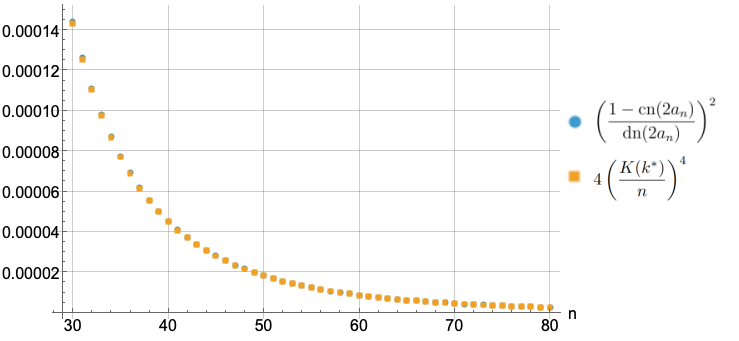}
    \caption{Asymptotics  of the value from Theorem \ref{Thm:SixthDerivativeRest}, for $30\leq n\leq80$.}
\end{figure}
\end{proof}

\section{Relation between the optimal constants and kernels associated to $\nabla^k$ and $\nabla^{2k}$}\label{SectionK-2K}

This section is devoted to establish Theorem \ref{Thm:RelationK-2K}. Our argument essentially relies on Fej\'er-Riesz theorem to represent the Fourier transform of a kernel in $\mathcal{F}_{n,2}$ as the square of the modulus of a polynomial. Nevertheless, a difficulty arises as the coefficients of this polynomial are complex numbers, and thus we need to show that the optimal constant does not decrease if we allow the kernels to be complex-valued. This result is stated in the following lemma.

\begin{lemma} \label{Lemma:ComplexKernels}
    Let $C_{k,n,1}$ be the constant defined in \eqref{Cons:Ckn1}. Then we have the inequality 
    \begin{equation*}
        \sup_{0 \neq f \in \ell^2(\mathbb{Z})} \frac{\| \nabla^k(u\ast f) \|_{\ell^2(\mathbb{Z})}}{\| f \|_{\ell^2(\mathbb{Z})}} \geq C_{k,n,1}
    \end{equation*}
    for all kernels $u : \{-n,\dots,n\} \to \mathbb{C}$ with normalization $\sum_{j=-n}^{n}u(j)=1$.
\end{lemma}
\begin{proof}
    By \cite[Lemma 4]{Richardson2026}, we know that 
    \begin{equation*}
        \sup_{0 \neq f \in \ell^2(\mathbb{Z})} \frac{\| \nabla^k(v\ast f) \|_{\ell^2(\mathbb{Z})}}{\| f \|_{\ell^2(\mathbb{Z})}} \geq C_{k,n,1}
    \end{equation*}
    for all kernels $v : \{-n,\dots,n\} \to \mathbb{R}$ with normalization $\sum_{j=-n}^{n}v(j)=1$. Let $u : \{-n,\dots,n\} \to \mathbb{C}$ be such that $\sum_{j=-n}^{n} u(j) = 1$, and let us consider $v = \text{Re}(u) = (u + \overline{u})/2$. It follows that $v$ is a real-valued kernel with normalization $\sum_{j=-n}^{n}v(j)=1$. By means of the triangle inequality, we have
    \begin{equation*}
        \begin{aligned}
            \frac{\| \nabla^k(v\ast f) \|_{\ell^2(\mathbb{Z})}}{\| f \|_{\ell^2(\mathbb{Z})}} &= \frac{1}{2} \frac{\| \nabla^k((u + \overline{u})\ast f) \|_{\ell^2(\mathbb{Z})}}{\| f \|_{\ell^2(\mathbb{Z})}}\\ &\leq \frac{1}{2} \left( \frac{\| \nabla^k(u\ast f) \|_{\ell^2(\mathbb{Z})}}{\| f \|_{\ell^2(\mathbb{Z})}} + \frac{\| \nabla^k( \overline{u} \ast f) \|_{\ell^2(\mathbb{Z})}}{\| f \|_{\ell^2(\mathbb{Z})}} \right)\\ &= \frac{1}{2} \left( \frac{\| \nabla^k(u\ast f) \|_{\ell^2(\mathbb{Z})}}{\| f \|_{\ell^2(\mathbb{Z})}} + \frac{\| \nabla^k( u \ast \overline{f}) \|_{\ell^2(\mathbb{Z})}}{\| \overline{f} \|_{\ell^2(\mathbb{Z})}} \right)
        \end{aligned}
    \end{equation*}
    for all $0 \neq f \in \ell^2(\mathbb{Z})$. As a result,
    \begin{equation*}
        \frac{\| \nabla^k(u\ast f) \|_{\ell^2(\mathbb{Z})}}{\| f \|_{\ell^2(\mathbb{Z})}} \geq \frac{\| \nabla^k(v\ast f) \|_{\ell^2(\mathbb{Z})}}{\| f \|_{\ell^2(\mathbb{Z})}} \quad \text{or} \quad \frac{\| \nabla^k( u \ast \overline{f}) \|_{\ell^2(\mathbb{Z})}}{\| \overline{f} \|_{\ell^2(\mathbb{Z})}} \geq \frac{\| \nabla^k(v\ast f) \|_{\ell^2(\mathbb{Z})}}{\| f \|_{\ell^2(\mathbb{Z})}}
    \end{equation*}
    for every $0 \neq f \in \ell^2(\mathbb{Z})$. Hence,
    \begin{equation*}
        \sup_{0 \neq f \in \ell^2(\mathbb{Z})} \frac{\| \nabla^k(u\ast f) \|_{\ell^2(\mathbb{Z})}}{\| f \|_{\ell^2(\mathbb{Z})}} \geq \sup_{0 \neq f \in \ell^2(\mathbb{Z})} \frac{\| \nabla^k(v\ast f) \|_{\ell^2(\mathbb{Z})}}{\| f \|_{\ell^2(\mathbb{Z})}}  \geq C_{k,n,1}.
    \end{equation*}
\end{proof}

Additionally, we need to show the existence of a unique normalized symmetric kernel $u$ supported in $\{-n,\dots,n\}$ that attains $C_{k,n,1}$. This problem is equivalent to minimizing $\|(1-x)^{k/2} p(x)\|_{L^{\infty}([-1,1])}$ among all polynomials $p$ of degree at most $n$ such that $p(1)=1$. The next lemma guarantees the existence and uniqueness of such a polynomial, and also provides an oscillatory behavior for it. This lemma may be established adapting the ideas from \cite[Section 2]{NovelloSchiefermayrZinchenko2021}, where it is shown the existence and uniqueness of an extremal polynomial for the corresponding problem for monic polynomials. 
The key difference is that for our purposes we replace the space of polynomials of degree at most $n-1$ in Kolmogorov's criterion \cite[Theorem 3]{NovelloSchiefermayrZinchenko2021} by the space
\begin{equation*}
    V_n = \{p \in \mathcal{P}_n : p(1)=0\},
\end{equation*}
where $\mathcal{P}_n$ denotes the space of polynomials of degree at most $n$.

\begin{lemma}\label{lemma:WeightedCheb}
    Let $w : [-1,1] \to [0,\infty)$ be a continuous function that takes non-zero values on at least $n+1$ points of $[-1,1]$. Then:
    \begin{enumerate}
        \item There exists a unique real polynomial $T_{n,w}$ of degree at most $n$ with $T_{n,w}(1)=1$ that minimizes $\|w(x) p(x)\|_{L^{\infty}([-1,1])}$ among all polynomials $p$ of degree at most $n$ satisfying $p(1) = 1$.
        \item Let $P_n$ be a real polynomial of degree at most $n$ such that $P_n(1) = 1$. Then $P_n = T_{n,w}$ if and only if there exist points $x_0 < x_1 < \dots < x_n$ on $[-1,1]$ such that
        \begin{equation*}
            w(x_j) P_n(x_j) = (-1)^{n-j}\|w(x)P_n(x)\|_{L^{\infty}([-1,1])}
        \end{equation*}
        for all $j \in \{0,1,\dots,n\}$.
    \end{enumerate}
\end{lemma}

Finally, a combination of Lemmas \ref{Lemma:ComplexKernels} and \ref{lemma:WeightedCheb} together with Fej\'er-Riesz theorem allows us to establish Theorem \ref{Thm:RelationK-2K}.

\begin{proof}[Proof of Theorem \ref{Thm:RelationK-2K}]
    Let $\mathcal{K}_{n,1}$ be the family of polynomials $p$ of degree at most $n$ with real coefficients such that $p(1)=1$, and let $\mathcal{K}_{n,2}$ be the family of those polynomials in $\mathcal{K}_{n,1}$ which are non-negative on $[-1,1]$. By applying the Fourier transform in $\mathbb{Z}$ together with Plancherel's theorem as in \eqref{Sup1Deriv}, we deduce that
    \begin{equation*}
        C_{k,n,1}^2 = \inf_{u \in \mathcal{F}_{n,1}}\|(e^{i\xi} - 1)^{2k} \widehat{u}(\xi)^2\|_{L^{\infty}(\mathbb{T})} = 2^k \inf_{p \in \mathcal{K}_{n,1}} \|(1-x)^k p(x)^2\|_{L^{\infty}([-1,1])}
    \end{equation*}
    and
    \begin{equation*}
        C_{2k,2n,2} = \inf_{u \in \mathcal{F}_{2n,2}}\|(e^{i\xi} - 1)^{2k} \widehat{u}(\xi)\|_{L^{\infty}(\mathbb{T})} =  2^k \inf_{q \in \mathcal{K}_{2n,2}} \|(1-x)^k q(x)\|_{L^{\infty}([-1,1])}.
    \end{equation*}
    Observe that $\{p^2 : p \in  \mathcal{K}_{n,1}\} \subseteq  \mathcal{K}_{2n,2}$, which implies that $C_{2k,2n,2} \leq C_{k,n,1}^2$.

    On the other hand, let $(u_j)$ be a sequence in $\mathcal{F}_{2n,2}$ such that
    \begin{equation}\label{SequenceKernels}
        \|(e^{i\xi} - 1)^{2k} \widehat{u_j}(\xi)\|_{L^{\infty}(\mathbb{T})} \to C_{2k,2n,2}
    \end{equation}
    as $j \to +\infty$. Applying Fej\'er-Riesz theorem, for each positive integer $j$ there exists a polynomial $Q_j(z) = \sum_{m=0}^{2n}b_{m}^{(j)} z^m$ with no zeros in $\mathbb{D}$ such that $\widehat{u_j}(\xi) = |Q_j(e^{i \xi})|^2$ for all $\xi \in \mathbb{T}$. Due to the normalization assumption on every $u_j$, we get that for each $j$ there exists $\varphi_j \in \mathbb{T}$ such that $Q_j(1) = e^{i \varphi_j}$. For each positive integer $j$, let us consider the kernel $v_j : \{-n,\dots,n\} \to \mathbb{C}$ given by $v_j(m) = b_{m+n}^{(j)} e^{-i\varphi_j}$ for any $m \in \{-n,\dots,n\}$. This means that $\widehat{v_j}(\xi) = e^{-i\varphi_j} e^{-in\xi}Q_j(e^{i\xi})$ for all $\xi \in \mathbb{T}$, and thus $\widehat{v_j}(0) = 1$ for any $j$. By Lemma \ref{Lemma:ComplexKernels}, we have that $\|(e^{i\xi} - 1)^{k} \widehat{v_j}(\xi)\|_{L^{\infty}(\mathbb{T})} \geq C_{k,n,1}$ for all $j$. Moreover, it follows from \eqref{SequenceKernels} that 
    \begin{equation*}
        \|(e^{i\xi} - 1)^{k} \widehat{v_j}(\xi)\|_{L^{\infty}(\mathbb{T})}^2 \to C_{2k,2n,2}.
    \end{equation*}
    Thus, $C_{2k,2n,2} \geq C_{k,n,1}^2$. This establishes that $C_{k,n,1}^2 = C_{2k,2n,2}$.

    Next, by Lemma \ref{lemma:WeightedCheb}, we know that there exists a unique polynomial $p \in \mathcal{K}_{n,1}$ such that 
    \begin{equation*}
        \|(1-x)^{k/2}p(x)\|_{L^{\infty}([-1,1])} = \inf_{q \in \mathcal{K}_{n,1}} \|(1-x)^{k/2} q(x)\|_{L^{\infty}([-1,1])}.
    \end{equation*}
    Arguing as in the proof of Theorem \ref{Thm:FourthDerivativeRest}, there exists a unique symmetric kernel $u : \{-n,\dots,n\} \to \mathbb{R}$ for which $p_u = p$ ($p_u$ is defined as in \eqref{PolV}). Since $\widehat{u}(0) = p_u(1) = p (1) = 1$, it follows that $u \in \mathcal{F}_{n,1}$. Thus,
    \begin{equation}\label{CAttained}
    \begin{aligned}
        C_{k,n,1} &= 2^{k/2}\inf_{q \in \mathcal{K}_{n,1}} \|(1-x)^{k/2} q(x)\|_{L^{\infty}([-1,1])}\\ &= 2^{k/2} \|(1-x)^{k/2}p_u(x)\|_{L^{\infty}([-1,1])}\\ &= \|(e^{i\xi}-1)^k \widehat{u}(\xi)\|_{L^{\infty}(\mathbb{T})}\\ &= \sup_{0 \neq f \in \ell^2(\mathbb{Z})} \frac{\|\nabla^{k} (u\ast f)\|_{\ell^2(\mathbb{Z})}}{\|f\|_{\ell^2(\mathbb{Z})}}.
    \end{aligned}
    \end{equation}
    Notice that if there is another kernel $v \in \mathcal{F}_{n,1}$ satisfying
    \begin{equation*}
        C_{k,n,1}=\sup_{0 \neq f \in \ell^2(\mathbb{Z})} \frac{\|\nabla^{k} (v\ast f)\|_{\ell^2(\mathbb{Z})}}{\|f\|_{\ell^2(\mathbb{Z})}},
    \end{equation*}
    then $p_v \in \mathcal{K}_{n,1}$ and
    \begin{equation*}
        \|(1-x)^{k/2}p_v(x)\|_{L^{\infty}([-1,1])} = \inf_{q \in \mathcal{K}_{n,1}} \|(1-x)^{k/2} q(x)\|_{L^{\infty}([-1,1])},
    \end{equation*}
    which implies that $p_v = p = p_u$. As a result, $u$ is the only kernel in $\mathcal{F}_{n,1}$ for which \eqref{CAttained} holds. Moreover, observe that
    \begin{equation}\label{CAttained2}
    \begin{aligned}
        C_{2k,2n,2} &= C_{k,n,1}^2\\ &= \|(e^{i\xi}-1)^{2k} \widehat{u}(\xi)^2\|_{L^{\infty}(\mathbb{T})}\\ &= \|(e^{i\xi}-1)^{2k} (u\ast u)^{\wedge}(\xi)\|_{L^{\infty}(\mathbb{T})}\\ &= \sup_{0 \neq f \in \ell^2(\mathbb{Z})} \frac{\|\nabla^{2k} ((u\ast u)\ast f)\|_{\ell^2(\mathbb{Z})}}{\|f\|_{\ell^2(\mathbb{Z})}}.
    \end{aligned}
    \end{equation}
    The oscillatory behavior of $(1-x)^{k/2}p_u$ (see Lemma \ref{lemma:WeightedCheb}) yields that $(1-x)^k p_{u \ast u} = ((1-x)^{k/2}p_u)^2$ has $2n+1$ interlaced equioscillation points in the interval $[-1,1)$. By applying an equioscillation argument similar to that of Proposition \ref{Prop:SquarePolIneq}, we deduce that $p_{u\ast u}$ is the only polynomial in $\mathcal{K}_{2n,2}$ satisfying
    \begin{equation*}
        \|(1-x)^{k}p_{u\ast u}(x)\|_{L^{\infty}([-1,1])} = \inf_{q \in \mathcal{K}_{2n,2}} \|(1-x)^{k} q(x)\|_{L^{\infty}([-1,1])}.
    \end{equation*}
    Hence, $u\ast u$ is the only kernel in $\mathcal{F}_{2n,2}$ satisfying \eqref{CAttained2}.
\end{proof}

\section{Acknowledgements}
J.M. was partially supported by the AMS Stefan Bergman Fellowship and the Simons Foundation Grant $\# 453576$. 

\bibliographystyle{acm}

\bibliography{bibl}

@book {RahmanSchmeisserBook,
    AUTHOR = {Rahman, Q. I. and Schmeisser, G.},
     TITLE = {Analytic theory of polynomials},
    SERIES = {London Mathematical Society Monographs. New Series},
    VOLUME = {26},
 PUBLISHER = {The Clarendon Press, Oxford University Press, Oxford},
      YEAR = {2002},
     PAGES = {xiv+742},
      ISBN = {0-19-853493-0},
   MRCLASS = {30C10 (00A05 11C08 12D10 30C15 31-02 41A05)},
  MRNUMBER = {1954841},
MRREVIEWER = {Bl.\ Sendov},
}

@book {GrenanderSzego1958,
    AUTHOR = {Grenander, Ulf and Szeg\"o, Gabor},
     TITLE = {Toeplitz forms and their applications},
    SERIES = {California Monographs in Mathematical Sciences},
 PUBLISHER = {University of California Press, Berkeley-Los Angeles},
      YEAR = {1958},
     PAGES = {vii+245},
   MRCLASS = {60.00 (15.00)},
  MRNUMBER = {94840},
MRREVIEWER = {M.\ Lo\`eve},
}

@article {Levedev1994,
    AUTHOR = {Lebedev, V. I.},
     TITLE = {Zolotarev polynomials and extremum problems},
   JOURNAL = {Russian J. Numer. Anal. Math. Modelling},
  FJOURNAL = {Russian Journal of Numerical Analysis and Mathematical
              Modelling},
    VOLUME = {9},
      YEAR = {1994},
    NUMBER = {3},
     PAGES = {231--263},
      ISSN = {0927-6467,1569-3988},
   MRCLASS = {41A10 (65K10)},
  MRNUMBER = {1285110},
       DOI = {10.1515/rnam.1994.9.3.231},
       URL = {https://doi-org.ezproxy.lib.vt.edu/10.1515/rnam.1994.9.3.231},
}

@article {KravitzSteinerberger2021,
    AUTHOR = {Kravitz, Noah and Steinerberger, Stefan},
     TITLE = {The smoothest average: {D}irichlet, {F}ej\'er and {C}hebyshev},
   JOURNAL = {Bull. Lond. Math. Soc.},
  FJOURNAL = {Bulletin of the London Mathematical Society},
    VOLUME = {53},
      YEAR = {2021},
    NUMBER = {6},
     PAGES = {1801--1815},
      ISSN = {0024-6093,1469-2120},
   MRCLASS = {42A85 (33C45 42A38 65D10)},
  MRNUMBER = {4386040},
MRREVIEWER = {Rita\ Correia\ Guerra},
       DOI = {10.1112/blms.12543},
       URL = {https://doi-org.ezproxy.lib.vt.edu/10.1112/blms.12543},
}

@article {Richardson2026,
    AUTHOR = {Richardson, Sean},
     TITLE = {A {S}harp {F}ourier {I}nequality and the {E}panechnikov
              {K}ernel},
   JOURNAL = {J. Fourier Anal. Appl.},
  FJOURNAL = {The Journal of Fourier Analysis and Applications},
    VOLUME = {32},
      YEAR = {2026},
    NUMBER = {1},
     PAGES = {Paper No. 14},
      ISSN = {1069-5869,1531-5851},
   MRCLASS = {42A05 (26D15 65T50)},
  MRNUMBER = {5013391},
       DOI = {10.1007/s00041-025-10215-1},
       URL = {https://doi-org.ezproxy.lib.vt.edu/10.1007/s00041-025-10215-1},
}

@article {Lax1944,
    AUTHOR = {Lax, Peter D.},
     TITLE = {Proof of a conjecture of {P}. {E}rd\"os on the derivative of a
              polynomial},
   JOURNAL = {Bull. Amer. Math. Soc.},
  FJOURNAL = {Bulletin of the American Mathematical Society},
    VOLUME = {50},
      YEAR = {1944},
     PAGES = {509--513},
      ISSN = {0002-9904},
   MRCLASS = {41.1X},
  MRNUMBER = {10731},
MRREVIEWER = {A.\ C.\ Offord},
       DOI = {10.1090/S0002-9904-1944-08177-9},
       URL = {https://doi-org.ezproxy.lib.vt.edu/10.1090/S0002-9904-1944-08177-9},
}

@article {Steinerberger2021,
    AUTHOR = {Steinerberger, Stefan},
     TITLE = {Fourier uncertainty principles, scale space theory and the
              smoothest average},
   JOURNAL = {Math. Res. Lett.},
  FJOURNAL = {Mathematical Research Letters},
    VOLUME = {28},
      YEAR = {2021},
    NUMBER = {6},
     PAGES = {1851--1874},
      ISSN = {1073-2780,1945-001X},
   MRCLASS = {42A38 (26A33 33C05 33C10 49J35 49S05)},
  MRNUMBER = {4477676},
MRREVIEWER = {Waseem\ Z.\ Lone},
       DOI = {10.4310/mrl.2021.v28.n6.a9},
       URL = {https://doi-org.ezproxy.lib.vt.edu/10.4310/mrl.2021.v28.n6.a9},
}

@article {ErdosSzego1942,
    AUTHOR = {Erd\"os, P. and Szeg\"o, G.},
     TITLE = {On a problem of {I}. {S}chur},
   JOURNAL = {Ann. of Math. (2)},
  FJOURNAL = {Annals of Mathematics. Second Series},
    VOLUME = {43},
      YEAR = {1942},
     PAGES = {451--470},
      ISSN = {0003-486X},
   MRCLASS = {41.1X},
  MRNUMBER = {6783},
MRREVIEWER = {A.\ C.\ Schaeffer},
       DOI = {10.2307/1968803},
       URL = {https://doi-org.ezproxy.lib.vt.edu/10.2307/1968803},
}

@misc{Bernstein1912,
 author = {Bernstein, S.},
 title = {Sur l'ordre de la meilleure approximation des fonctions continues par des polynomes de degr{\'e} donn{\'e}.},
 year = {1912},
 language = {French},
 howpublished = {Belg. {Mem}. in {{\(4^\circ\)}} (2) 4, 104 {S}.},
 zbMATH = {2616863},
 JFM = {45.0633.03}
}

@misc{NIST:DLMF,
         key = "{\relax DLMF}",
       title = "{\it NIST Digital Library of Mathematical Functions}",
howpublished = "\url{https://dlmf.nist.gov/}, Release 1.2.6 of 2026-03-15",
         url = "https://dlmf.nist.gov/",
        note = "F.~W.~J. Olver, A.~B. {Olde Daalhuis}, D.~W. Lozier, B.~I. Schneider,
                R.~F. Boisvert, C.~W. Clark, B.~R. Miller, B.~V. Saunders,
                H.~S. Cohl, and M.~A. McClain, eds."}

@article {BojanovNaidenov2010,
    AUTHOR = {Bojanov, Borislav and Naidenov, Nikola},
     TITLE = {On oscillating polynomials},
   JOURNAL = {J. Approx. Theory},
  FJOURNAL = {Journal of Approximation Theory},
    VOLUME = {162},
      YEAR = {2010},
    NUMBER = {10},
     PAGES = {1766--1787},
      ISSN = {0021-9045,1096-0430},
   MRCLASS = {33C45 (30C10 41A10 41A17 42C05)},
  MRNUMBER = {2728045},
MRREVIEWER = {Yamilet\ Quintana},
       DOI = {10.1016/j.jat.2010.04.009},
       URL = {https://doi-org.ezproxy.lib.vt.edu/10.1016/j.jat.2010.04.009},
}

@article {Peherstorfer2009,
    AUTHOR = {Peherstorfer, Franz},
     TITLE = {Extremal problems of {C}hebyshev type},
   JOURNAL = {Proc. Amer. Math. Soc.},
  FJOURNAL = {Proceedings of the American Mathematical Society},
    VOLUME = {137},
      YEAR = {2009},
    NUMBER = {7},
     PAGES = {2351--2361},
      ISSN = {0002-9939,1088-6826},
   MRCLASS = {41A29 (33C45 41A60)},
  MRNUMBER = {2495269},
MRREVIEWER = {Andrei\ Mart\'inez Finkelshtein},
       DOI = {10.1090/S0002-9939-09-09771-8},
       URL = {https://doi-org.ezproxy.lib.vt.edu/10.1090/S0002-9939-09-09771-8},
}

@article {Peherstorfer1991,
    AUTHOR = {Peherstorfer, Franz},
     TITLE = {On the connection of {P}osse's {$L_1$}- and {Z}olotarev's
              maximum-norm problem},
   JOURNAL = {J. Approx. Theory},
  FJOURNAL = {Journal of Approximation Theory},
    VOLUME = {66},
      YEAR = {1991},
    NUMBER = {3},
     PAGES = {288--301},
      ISSN = {0021-9045,1096-0430},
   MRCLASS = {41A10},
  MRNUMBER = {1122290},
MRREVIEWER = {D.\ C.\ Gilles},
       DOI = {10.1016/0021-9045(91)90032-6},
       URL = {https://doi-org.ezproxy.lib.vt.edu/10.1016/0021-9045(91)90032-6},
}

@incollection {NovelloSchiefermayrZinchenko2021,
    AUTHOR = {Novello, Galen and Schiefermayr, Klaus and Zinchenko, Maxim},
     TITLE = {Weighted {C}hebyshev polynomials on compact subsets of the
              complex plane},
 BOOKTITLE = {From operator theory to orthogonal polynomials, combinatorics,
              and number theory---a volume in honor of {L}ance
              {L}ittlejohn's 70th birthday},
    SERIES = {Oper. Theory Adv. Appl.},
    VOLUME = {285},
     PAGES = {357--370},
 PUBLISHER = {Birkh\"auser/Springer, Cham},
      YEAR = {[2021] \copyright 2021},
      ISBN = {978-3-030-75424-2; 978-3-030-75425-9},
   MRCLASS = {41A50 (30C10 30E10)},
  MRNUMBER = {4367475},
MRREVIEWER = {Brian\ Simanek},
}

\end{document}